\documentclass{article}
\usepackage{arxiv}

\usepackage[utf8]{inputenc} 
\usepackage[T1]{fontenc}    

\usepackage{ bbold }
\usepackage{graphics}

\usepackage{pst-node}
\usepackage{tikz-cd}
\usepackage{enumitem}
\usepackage{wrapfig,color}

\usepackage{amsmath}
\usepackage{amssymb}
\usepackage{amsthm}
\usepackage{xspace}
\usepackage[normalem]{ulem}
\usepackage{graphicx}
\usepackage{epsfig}
\usepackage{ifpdf}
\usepackage{soul,hyperref}
\usepackage{latexsym}
\usepackage[mathscr]{euscript}
\usepackage{xspace}
\usepackage{color}
\usepackage{makeidx}
\usepackage{wrapfig,color}
\usepackage{stackrel}
\usepackage[linesnumbered,algoruled,boxed,lined]{algorithm2e}
\usepackage[noend]{algpseudocode}
 \usepackage[toc,page]{appendix}

\usepackage[all]{xy}

\usepackage{mathtools}

\usepackage{extarrows}
\usepackage{kbordermatrix}
\usepackage{todonotes}

\renewcommand{\headeright}{}
\renewcommand{\undertitle}{}
\renewcommand{\shorttitle}{}

\title{Generalized Persistence Algorithm for Decomposing  Multiparameter Persistence Modules}

\author{
    Tamal K. Dey \\
\texttt{tamaldey@purdue.edu} 
\And
Cheng Xin\\
\texttt{xinc@purdue.edu} 
}

\date{Department of Computer Science\\ Purdue University} 

\newtheorem{construction}{Construction}[section]

\theoremstyle{plain}
\newtheorem{theorem}{Theorem}[section]
\newtheorem{proposition}[theorem]{\bf Proposition}
\newtheorem{example}{Example}
\theoremstyle{definition}
\newtheorem{definition}{Definition}[section]
\newtheorem{fact}{Fact}

\newtheorem{remark}{Remark}[section]

\newtheorem*{theorem*}{Theorem}
\newtheorem*{proposition*}{Proposition}
\newtheorem*{lemma*}{Lemma}
\newtheorem*{note*}{Note}
\newtheorem*{remark*}{Remark}
\newtheorem*{notation*}{Notation}

 \newcommand{\propositionofref}{}

 \newcommand{\theoremofref}{}

\newtheorem{corollary}[theorem]{Corollary}
\newtheorem{observation}[theorem]{Observation}

\newcommand{\Int}{\mathbb{Z}}
\newcommand{\Real}{\mathbb{R}}
\newcommand{\N}{\mathbb{N}}
\newcommand{\A}{\mathbf{A}}
\newcommand{\X}{\mathbf{X}}
\newcommand{\Y}{\mathbf{Y}}
\newcommand{\I}{\mathbf{I}}
\newcommand{\ttt}{\mathbf{t}}

\newcommand{\uu}{\mathbf{u}}
\newcommand{\vv}{\mathbf{v}}

\newcommand{\BS}{\mathbf{S}}
\newcommand{\LL}{\mathbf{P}}
\newcommand{\RR}{\mathbf{Q}}

\newcommand{\Id}{\mathbb{1}}

\newcommand*{\rom}[1]{\expandafter\@slowromancap\romannumeral #1@}
\newcommand{\Img}{\mathrm{im\,}}

\newcommand{\dm}{\mathrm{dm}}

\newcommand{\field}[1] {{\mathbb{#1}}}
\newcommand{\col}{{\sf Col}}
\newcommand{\row}{{\sf Row}}
\newcommand{\low}{{\sf Low}}
\newcommand{\lin}{{\sf Lin}}
\newcommand{\B}{\mathcal{B}}
\newcommand{\rowop}{{\sf Rowop}}
\newcommand{\colop}{{\sf Colop}}

\newcommand{\cok}{\mathrm{coker}}

\newcommand{\gr}{\mathrm{gr}}
\newcommand{\dotr}{\mbox{$\boldsymbol{\cdot}$}}

\definecolor{darkred}{rgb}{1, 0.1, 0.3}
\definecolor{darkblue}{rgb}{0.1, 0.1, 1}

\begin{document}

\maketitle
\begin{abstract}
The classical persistence algorithm computes the unique decomposition of a persistence
module implicitly
given by an input simplicial filtration.
Based on matrix reduction, this algorithm is a cornerstone of the emergent area of topological data analysis. Its input is a simplicial filtration defined over the
integers $\mathbb{Z}$ giving rise to a $1$-parameter persistence module. It has been recognized that multiparameter version of persistence modules given by
simplicial filtrations over $d$-dimensional integer grids
$\mathbb{Z}^d$ is equally or perhaps more important in data science applications. However, in the multiparameter setting, one of the main challenges is that topological summaries based on algebraic structure such as decompositions and bottleneck distances
cannot be as efficiently computed as in the $1$-parameter case because there is no known extension of the persistence algorithm to
multiparameter persistence modules. 
We present an efficient algorithm to compute the unique decomposition of a finitely presented persistence module $M$ defined over the multiparameter $\mathbb{Z}^d$. 
The algorithm first assumes that the
module is presented with a set of $N$ generators and relations that are \emph{distinctly graded}.
Based on a generalized matrix reduction technique it runs in $O(N^{2\omega+1})$ time where $\omega<2.373$ is the exponent for matrix multiplication. This is much better than
the well known algorithm called Meataxe which runs in $\tilde{O}(N^{6(d+1)})$ time on such an input. In practice, persistence modules are usually induced by simplicial filtrations. With such an input consisting of $n$ simplices, our algorithm runs in $O(n^{(d-1)(2\omega + 1)})$ time for $d\geq 2$. For the special case of zero dimensional homology, it runs in time $O(n^{2\omega +1})$.


\end{abstract}

\section{Introduction} \label{sec: intro}
Persistence modules defined over a single parameter such as $\mathbb{Z}$ or $\mathbb{R}$ have 
become a central object of study in topological data analysis (TDA). It is an indexed set of vector
spaces connected by linear maps most commonly arising from applying a homology functor to
a simplicial filtration--another well known construct in TDA. Under some mild conditions~\cite{oudot2015persistence},
such a module decomposes uniquely into interval modules called \emph{bars}. These bars or its equivalent 
persistence diagrams encode the input module completely. Naturally, computing these bars 
from an input persistence module efficiently becomes an important endeavor in TDA. Starting with the
persistence algorithm~\cite{edelsbrunner2000topological}, a number of improvements and extensions have been proposed for computing the bar decompositions in
the single parameter case. However, the problem in the multi parameter case 
has not received as much attention. 
Other than some specific cases~\cite{BLO20,cai2020elder,CO20,DX18},
the only known algorithm for the purpose can be derived from the so-called Meataxe algorithm which 
applies to much more general modules than we consider in TDA at the expense of high computational cost.
Sacrificing this generality and still encompassing a large class
of modules that appear in TDA, we can design a much more efficient algorithm.
Specifically, we present an algorithm that can decompose a finitely
presented module (unique decomposition is guaranteed by Krull-Schmidt
theorem~\cite{Atiyah1956}) with
a time complexity that is far better than the Meataxe algorithm though we loose the generality as
the module needs to be \emph{distinctly graded}, that is, 
no two generators and no two relations of the module have the same grade. 
If this condition is not satisfied, a simple observation 
implies that the algorithm
still produces an output that can be viewed as a decomposition of a module close
under the interleaving distance.

For measuring algorithmic efficiency, it is imperative to specify how the input module is presented.
Assuming an index set of size $m$ and vector spaces of dimension $O(s)$, 
a one-parameter persistence module can be presented by a set of $O(s)\times O(s)$ matrices
each representing a linear map $M_i\rightarrow M_{i+1}$ between two consecutive vector spaces $M_i$ and $M_{i+1}$. This input format is costly as it takes $O(ms^2)$
 space ($O(s^2)$-size matrix for each index) and also does not appear to offer any benefit in time complexity for computing the bars. 
An alternative presentation is obtained by considering the persistence module
as a graded module over a polynomial ring $\mathbb{k}[t]$ and presenting it with the so-called \emph{generators} 
$\{g_i\}$ of the module and \emph{relations} $\{\sum_{i}\alpha_ig_i=0\,\mid\,\alpha_i\in\mathbb{k}[t]\}$ 
among them. A presentation matrix encoding the
relations in terms of the generators characterizes the module completely. 
Then, a matrix reduction algorithm akin to the persistence algorithm~\cite{CEM06}
provides the desired decomposition. Figure~\ref{fig:oneparam}
illustrates the advantage of this presentation over the other costly presentation. 
In practice, when the $1$-parameter persistence module is given by an
implicit simplicial filtration, one can apply the matrix reduction algorithm directly on a boundary
matrix rather than first computing a presentation matrix from it and then decomposing it. If there
are $O(N)$ generators (creator simplices) and relations (destructor simplices), the algorithm
runs in $O(N^3)$ time with simple matrix reductions and in $O(N^\omega)$ time with more sophisticated
matrix multiplication techniques where $\omega<2.373$ is the exponent for matrix multiplication.

\begin{figure}[!htb]
        \centering
         \includegraphics[width=\linewidth, page=3]{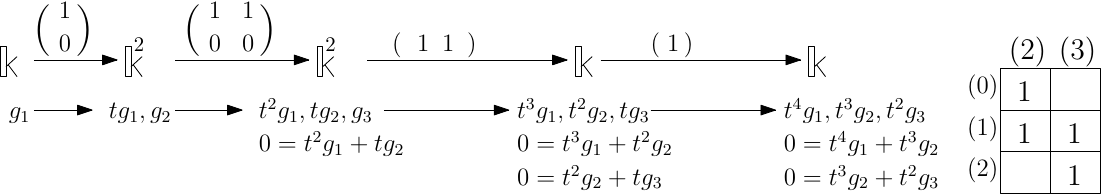}
         \caption{Costly presentation (top) vs. graded presentation (bottom,right). The top chain can be summarized by three generators $g_1, g_2, g_3$ at grades $(0), (1), (2)$ respectively, and two relations $0=t^2g_1+tg_2$, $0=t_2g_2+tg_3$ at grades $(2),(3)$ respectively. The grades of generators and relations are given by the first times they appear in the chain. Finally, these information can be summarized succinctly by the presentation matrix on the right. 
         }\label{fig:oneparam}
\end{figure}

The Meataxe algorithm for multiparameter persistence modules follows the costly approach
analogous to the one in the one-parameter case that expects the presentation of 
each individual linear map explicitly.
In particular, It expects the input $d$-parameter module $M$ over a finite
subset of $\mathbb{Z}^d$ to be given as a large matrix in $\field{k}^{D\times D}$ with entries in a fixed field $\field{k}=\mathbb{F}_q$, where $D$
is the sum of dimensions of vector spaces over all points in $\mathbb{Z}^d$ supporting $M$. The time complexity of the Meataxe algorithm is $O(D^6\log q)$ ~\cite{holtmeataxe}. 
In general, $D$ might be quite large. It is not clear what is the most efficient way to transform an input that specifies generators and relations ( or a simplicial filtration)
to a representation matrix required by the Meataxe algorithm. A naive approach
is to consider the minimal sub-grid in $\mathbb{Z}^d$ that supports the non-trivial maps.
In the worst-case, with $N$ being the total number of generators and relations, one has to consider
$O({N\choose d})=O(N^d)$ grid points in 
$\Int^d$ each with a vector space of dimension $O(N)$.
Therefore, $D=O(N^{d+1})$ 
giving a worst-case time complexity of $O(N^{6(d+1)}\log q)$.
Even allowing approximation, the algorithm~\cite{holt1994testing} runs in  
$O(N^{3(d+1)}\log q)$ time.

We take the alternate approach where the module is treated as a finitely presented graded module over multivariate polynomial ring $R=\field{k}[t_1,\cdots,t_d]$ \cite{Corbet2018} and presented with
a set of generators and relations graded appropriately. The fact that the persistence modules in TDA can be 
modeled as a graded module studied in algebraic geometry and commutative algebra~\cite{eisenbud2005geometry,miller2004combinatorial}
was recognized in~\cite{carlsson2009computing,Carlsson2009,knudson2007refinement} and further studied in  ~\cite{lesnick2015theory,Lesnick_compute_minimal_present_19}. 
Given a presentation matrix encoding relations with generators,
our algorithm computes a diagonalization of the matrix giving a presentation of each module called
\emph{indecompsable} which the input module decomposes into. These indecomposables are the higher
dimensional analogues of the bars. Compared to the one-parameter case,
we have to cross two main barriers for computing the indecomposables. First, we need to allow row operations along with column operations for
reducing the input matrix. In one-parameter case, row operations become redundant because 
column operations already produce the bars. Second, differently from the one-parameter case,
we cannot allow all left-to-right column or bottom-to-top row operations for the matrix
reduction because the parameter space $\mathbb{Z}^d$, $d>1$, unlike $\mathbb{Z}$ only
induces a partial order on these operations. We show how these two
difficulties can be overcome by an incremental approach combined with a linearization trick.
Given a presentation matrix with a total of $N$ generators and relations that are graded
distinctly, our algorithm
runs in $O(N^{2\omega+1})$ time. Surprisingly, 
the complexity does not depend on the parameter $d$.
with time complexity $O(N^{2\omega+1})$ where $\omega<2.373$ is the matrix multiplication exponent.

In practice, we are often given a simplicial filtration instead of a presentation matrix
relating the generators of the induced persistence module. In this case, one has to
compute presentation matrices from the input filtration consisting of $n$ simplices.
For $2$-parameter persistence modules, we can compute a presentation matrix of size $O(n)\times O(n)$ using the algorithm of Lesnick and Wright~\cite{Lesnick_compute_minimal_present_19} in $O(n^3)$ time whereas for $d$-parameter persistence modules, we can adapt an algorithm of Skryzalin~\cite{skryzalin} to compute the presentation in $O(n^{d+1})$ time. 
For $d \geq 2$, this algorithm produces a presentation matrix of dimension $O(n^{d-1})\times O(n^{d-1})$.
 Therefore, with $N=O(n^{d-1})$, the decomposition algorithm takes $O(n^{(d-1)(2\omega+1)})$ time. Combining the costs for computing a presentation and its decomposition, the time complexity of our algorithm becomes $O(n^{(d-1)(2\omega+1)})$ for $d \geq 2$. The time complexity of the Meataxe algorithm remains the same, $O(n^{6(d+1)}\log q)$, with a simplicial filtration of $n$ simplices because the highest dimension of the vector space at each grid point is $O(n)$. Our algorithm is better than the Meataxe algorithm in this case too.


As a generalization of the traditional persistence algorithm, it is expected that our algorithm can be interpreted as computing invariants such as persistence diagrams~\cite{Cohen-Steiner2007} or barcodes~\cite{zomorodian2005computing}.
A roadblock to this goal is that $d$-parameter persistence modules do not have complete discrete  invariants for $d\geq 2$~\cite{Carlsson2009,lesnick2015theory} . Consequently, one needs to invent other invariants suitable for multiparameter persistence modules. A natural way to generalize the invariant in traditional persistent homology would be to consider the decomposition and take the discrete invariants in each indecomposable component. This gives us invariants which are no longer complete but still contain rich information.

We offer two interpretations of the output of our algorithm as two different invariants: \emph{persistent graded Betti numbers} as a generalization of persistence diagrams and \emph{blockcodes} as a generalization of barcodes. The persistent graded Betti numbers are linked to the graded Betti numbers studied in commutative algebra, which is introduced in TDA for multiparameter persistence modules in the work of~\cite{Carlsson2009} and~\cite{knudson2007refinement}. The bigraded Betti numbers are further studied in~\cite{Lesnick_compute_minimal_present_19}. By constructing the free resolution of a persistence module, we can compute its graded Betti numbers and then decompose them according to each indecomposable module, which results into the presistent graded Betti numbers. For each indecomposable, we apply the dimension function which is also known as the Hilbert function in commutative algebra
to summarize the graded Betti numbers for each indecomposalbe module. This constitutes a blockcode for the indecomposable module of the persistence module. The blockcode is a good vehicle for visualizing lower dimensional persistence modules such as $2$- or $3$-parameter persistence modules.

\subsection{Other related work}
Since it is known that there is no complete discrete invariant for multiparameter persistence, researchers have proposed various reasonable summaries that can be computed in practice. Among them the rank invariant proposed by Carlsson et al.~\cite{carlsson2009computing,Carlsson2009} is a popular one. Cerri et al.~\cite{Cerri_Betti_numbers} propose multiparameter persistent Betti number as a stable invariant. Lesnick and Wright introduce the computational tool of fibered barcode in~\cite{2015arXiv151200180L,Lesnick_compute_minimal_present_19} as an interactive vehicle to visualize the one-parameter restriction of biparameter persistence modules.

Another related line of work focuses on defining distances and their stabilities in the space of
multiparameter persistence modules.
The interleaving distance~\cite{Bjerkevik_complexity_intl_dist,Bjerkevik_intl_dist_nphard,bjerkevik2016stability,lesnick2015theory}, and multi-matching distance~\cite{Cerri_Betti_numbers,2018arXiv180106636C} 
are some of the work to mention a few. The relation between interleaving distance and bottleneck distance is studied in~\cite{bjerkevik2016stability,botnan2016algebraic,Emerson_19}. On the computational front, Dey and Xin showed that the bottleneck distance can be computed in polynomial time for the special cases of interval decomposable modules~\cite{DX18} though the general problem is proved to be NP-hard~\cite{Bjerkevik_complexity_intl_dist,Bjerkevik_intl_dist_nphard}.
A recent work of Kerber et al. shows that the matching distance~\cite{kerber2018exact} can be computed efficiently in polynomial time.

\subsection{Outline}
The rest of the paper is organized as follows. In Section 2, we introduce some background materials on persistence modules in the language of graded modules. In Section 3, we introduce the presentation of a persistence module and its presentation matrix which assists in computing the decomposition of persistence modules. The 1-1 correspondence between the decompositions of the persistence module and its presentation is a fundamental fact which is presented as our first main theorem. Based on this correspondence, we observe that two main computational problems need to be solved, (i) computing the decomposition of the presentation matrix, and (ii) constructing a valid presentation.
In Section 4, we handle the first problem by designing an algorithm for computing a decomposition of the presentation matrix. We observe that this problem can be transformed to a what we call generalized matrix reduction problem. Based on that, we propose an algorithm to solve this problem and prove the correctness of our algorithm, and illustrate it with an example. In Section 5, we introduce the strategies for the second problem of computing presentations and analyze the total time complexity for computing presentations together with matrix reduction.
In Section 6, we give two interpretations of the results of our decomposition of persistence modules as two different invariants, persistent graded Betti numbers as a generalization of persistence diagrams and blockcodes as a generalization of barcodes. In Section 7, we conclude with suggesting some future direction.

\section{Persistence modules}
We want to study the \emph{total decomposition} of a persistence module arising from a simplicial filtration in the multiparameter setting.
We first present some preliminary concepts from commutative algebra that lay the foundation
of this work. For more details on multiparameter persistent homology and commutative algebra, we refer the readers 
to~\cite{CM_rings_bruns1998,Carlsson2009,cox2006usingAG,miller2004combinatorial}.
Mainly, we need the concept of \emph{graded modules} because as in~\cite{Carlsson2009} we treat the familiar persistence modules in topological data analysis as graded modules.
Let $R=\field{k}[t_1, \cdots, t_d]$ be the $d$-variate Polynomial ring for some $d\in \Int_+$ with $\field{k}$ being a field. Throughout this paper, we assume coefficients are in $\field{k}$. Hence homology groups are vector spaces. 
\begin{definition}
A $\Int^d$-graded $R$-module (graded module in brief) is an $R$-module $M$ that is a direct sum of $\field{k}$-vector spaces $M_\mathbf{u}$ indexed by $\mathbf{u}\in\Int^d$, i.e. $M= \bigoplus_{\mathbf{u}}M_\mathbf{u}$,
such that the ring action satisfies that $\forall i, \forall\mathbf{u}\in \Int^d$, $t_i\cdot  M_\mathbf{u}\subseteq M_{\mathbf{u}+e_i}$, where $\{e_i\}_{i=1}^d$ is the standard basis in $\Int^d$.
\end{definition}
Another interpretation of graded module is that, for each $\mathbf{u}\in \Int^d$, the action of $t_i$ on $M_\mathbf{u}$ determines a linear map $t_i\bullet :M_\mathbf{u}\rightarrow M_{\mathbf{u}+e_i}$ by $(t_i\bullet)(m)=t_i\cdot m$. So, we can also describe a graded module equivalently as a collection of vectors spaces $\{M_\mathbf{u}\}_{\mathbf{u}\in \Int^d}$ with a collection of linear maps $\{t_i\bullet:M_\mathbf{u}\rightarrow M_{\mathbf{u}+e_i},\forall i, \forall \mathbf{u}\}$
 where
the commutative property
$(t_j\bullet)\circ(t_i\bullet)=(t_i\bullet)\circ(t_j\bullet)$ holds.
The commutative diagram in Figure~\ref{bigrade-fig} shows a graded module for $d=2$, also called a bigraded module.
\begin{figure}[ht!]
\begin{center}
\begin{tikzcd}
                      & \cdots   & \cdots   & \cdots  &   \\
\cdots \arrow[r]      & {M_{0,2}} \arrow[r] \arrow[u]                                       & {M_{1,2}} \arrow[r] \arrow[u]                                          & {M_{2,2}} \arrow[r] \arrow[u] & \cdots \\
\cdots\longrightarrow & {M_{0,1}} \arrow[u, "t_2\bullet" description] \arrow[r, "t_1\bullet"]                                                                                                                                                            & {M_{1,1}} \arrow[u] \arrow[r]                                          & {M_{2,1}} \arrow[u] \arrow[r] & \cdots \\
\cdots \arrow[r]      & {M_{0,0}} \arrow[r, "t_1\bullet" description] \arrow[u, "t_2\bullet" description] \arrow[ru, "t_1\bullet t_2\bullet" description, dashed] \arrow[uu, "t_2^2\bullet", dashed, bend left] \arrow[rr, "t_1^2\bullet"', dashed, bend right] & {M_{1,0}} \arrow[r, "t_1\bullet" description] \arrow[u, "t_2\bullet"'] & {M_{2,0}} \arrow[u] \arrow[r] & \cdots \\
                      & \cdots \arrow[u]                                                                                                                                                                                                                 & \uparrow                                                               & \cdots \arrow[u]              &
\end{tikzcd}
\end{center}
\caption{A graded $2$-parameter module. All sub-diagrams of maps and compositions of maps are commutative.}
\label{bigrade-fig}
\end{figure}

We call a graded module $M$ \emph{finitely generated}
if there exists a finite set of elements $\{g_1, \cdots, g_n\}\subseteq M$ such that each element $m\in M$ can be written as an $R$-linear combination of these elements, i.e. $m=\sum_{i=1}^n \alpha_i g_i$ with $\alpha_i\in R$. 
We call this set $\{g_i\}$ a \emph{generating set} of $M$.  
In this paper, we assume that all modules are finitely generated.
Such modules always admit a minimal generating set. 

\begin{definition}
A graded module morphism, called \emph{morphism} in short, between two modules $M$ and $N$ is defined as an $R$-linear map $f: M\rightarrow N$ preserving grades: $f(M_\mathbf{u})\subseteq N_{\mathbf{u}}, \forall \mathbf{u}\in \Int^d$. Equivalently, it can also be described as a collection of linear maps $\{f_\mathbf{u}: M_\mathbf{u}\rightarrow N_{\mathbf{u}}\}$ which gives the following commutative diagram for each $\mathbf{u}$ and $i$:
\[
\begin{tikzcd}
M_\mathbf{u} \arrow[r, "t_i"] \arrow[d, "f_\mathbf{u}"'] & M_{\mathbf{u}+e_i} \arrow[d, "f_{\mathbf{u}+e_i}"] \\
N_\mathbf{u} \arrow[r, "t_i"]                            & N_{\mathbf{u}+e_i}
\end{tikzcd}
\]
Two graded modules $M,N$ are isomorphic if there exist two morphisms $f:M\rightarrow N$ and $g:N\rightarrow M$ such that $g\circ f$ and $f\circ g$ are identity maps. 

\end{definition}

For a graded module $M$, define a shifted graded module $M_{\rightarrow \mathbf{u}}$ for some $\mathbf{u}\in \Int^d$ by requiring $(M_{\rightarrow \mathbf{u}})_{\mathbf{v}}=M_{\mathbf{v}-\mathbf{u}}$ for each $\mathbf{v}$. 

\begin{definition}[Free module]
We say a graded module is \emph{free} if it is isomorphic to the direct sum of a collection of $R_{\rightarrow \mathbf{u}_j}$'s for some $\mathbf{u}_j$'s in $\Int^d$, denoted as
$\bigoplus_{j}R_{\rightarrow \mathbf{u}_j}$. 
\end{definition}
\begin{definition}[homogeneous]
We say an element $m\in M$ is \emph{homogeneous} if $m\in M_\mathbf{u}$ for some $\mathbf{u}\in \Int^d$. We denote $\gr(m)=\mathbf{u}$ as the \emph{grade} of such homogeneous element.
To emphasize the grade of a homogeneous element, we also write $m^{\gr{(m)}}:=m$.
\end{definition}
A minimal generating set of a free module is called a \emph{basis}.
We usually further require that all the elements in a basis, also called \emph{generators}, are homogeneous.
For a free module $F\simeq \bigoplus_{j}R_{\rightarrow \mathbf{u}_j}$, $\{e_j:j=1,2,\cdots\}$ is a
homogeneous basis of $F$, where $e_j$ indicates the multiplicative identity in $R_{\rightarrow \mathbf{u}_j}$.
The generating set $\{e_j:j=1,2,\cdots\}$ is often referred to as the standard basis of
$\bigoplus_{j}R_{\rightarrow \mathbf{u}_j}=<\{e_j:j=1,2,\cdots\}>$.

A \emph{$d$-parameter persistence module} is a graded $R$-module 
obtained by applying the homology functor with some field $\field{k}$ on a $d$-parameter simplical filtration defined below. In this paper, it can be treated as a synonym for a $\mathbb{Z}^d$-graded $R$-module.
Formally, a ($d$-parameter) \emph{simplicial filtration} is a family of simplicial complexes $\{X_\mathbf{u}\}_{\mathbf{u}\in \Int^d}$ such that for each grade $\mathbf{u}\in \Int^d$ and each $i=1,\cdots, d$, $X_\mathbf{u}\subseteq X_{\mathbf{u}+e_i}$.
We obtain a
simplicial chain complex $(C_{\dotr}(X_\mathbf{u}), \partial_{\dotr})$ 
for each $X_\mathbf{u}$ 
in this simplicial filtration. 
For each chain complex $C_{\dotr}(X_\mathbf{u})$, we have the cycle spaces $Z_p(X_{\mathbf{u}})$'s and boundary spaces $B_p(X_{\mathbf{u}})$'s as kernels and images of boundary maps $\partial_p$'s respectively, and the homology group $H_p(X_{\mathbf{u}})=Z_p(X_{\mathbf{u}})/B_p(X_{\mathbf{u}})$ as the cokernel of the inclusion maps $B_p(X_{\mathbf{u}})\hookrightarrow Z_p(X_{\mathbf{u}})$. Taking $M_{\mathbf u}=H_p(X_{\mathbf u})$ and the linear maps $H_p(X_{\mathbf u})\rightarrow H_p(X_{\mathbf v})$ induced by inclusions $X_{\mathbf u}\subseteq X_{\mathbf v}$
define a $d$-parameter persistence module. More details of this construction
is given later in Section~\ref{sec:compte_presentation}.
For illustration purpose, we describe a working example of a $2$-parameter persistence module induced from a $2$-parameter simplicial filtration shown in Figure~\ref{fig:working_example}. We will use this example throughout to show its induced persistence module and computational results of our algorithm. Later in the context, when we mention an example without reference, we refer to this working example.
\begin{figure}
    \centering
    \includegraphics[page=1, width=0.5\textwidth]{./img/working_example.pdf}
    \caption{The working example on a $2$-parameter simplicial filtrations. Each square box indicates what is the current (filtered) simplical complex at the grade of the box. It has one connected component in $0$th homology groups at grades except $(0,0)$ and $(1,1)$, and has two connected components at grade $(1,1)$.}
    \label{fig:working_example}
\end{figure}

\begin{example}\label{ex:working_example}[Working example]
In practice, 
the most common simplical filtration is obtained from the sublevel sets $\{X_{\mathbf{u}}:=f^{-1}(-\infty, \uu]\}_{\mathbf{u}\in \Real^n}$ of a given (one-critical) filtration function $f:X\rightarrow\Int^d$ on a topological space represented by a simplicial complex $X$. 


For example, let the space $X$ be a simplicial $1$-complex with $0$-simplices consisting of three vertices, blue vertex $v_b$, red vertex $v_r$, and green vertex $v_g$, connected by three edges, blue edge $e_b$, red edge $e_r$, and green edge $e_g$ as $1$-simplices. Assign a filtration function $f:X\rightarrow \Int^2$ as follows:
\begin{align*}
f(v_b)=(0,1), f(v_r)=(1,0), f(v_g)=(1,1)\\
f(e_b)=(1,2), f(e_r)=(1,1), f(e_g)=(2,1)
\end{align*}

Based on this filtration function, the subcomplex $X_\uu$ for each $\uu\in \Int^2$ is illustrated in Figure~\ref{fig:working_example}. Take vertices as basis of each $C_0(X_\uu)$ and edges as basis of $C_1(X_\uu)$. Recall that to emphasize the grades, we denote $v_*^\uu\in C_0(X_\uu)$ to be the basic element in the vector space $C_0(X_\uu)$. All these $v_*^\uu$ are homogeneous element in the graded module $C_0(X)$.
For each vertex $v_*\in \{v_b, v_r, v_g\}$, there is a unique smallest grade $\gr(v_*)\triangleq f(v_*)$ such that $v_*^{\gr(v_*)}$ is a homogeneous basic element in $C_0(X_{\gr(v_*)})$ and $\uu'\ngeq\gr(v_*)\implies v_*\notin C_0(X_{\uu'})$. 
We call this grade $\gr(v_*)$ the \emph{birth time} of $v_*$. 
Then for all $\uu\geq\gr(v_*)$, by the definition of scaler multiplication of graded modules, 
$\ttt^{\uu-\gr(v_*)} v_*^{\gr(v_*)}=v_*^{\uu} \in C_0(X_{\uu})$ 
is the image of $v_*^{\gr(v_*)}$ under the inclusion map
$C_0(X_{\gr(v_*)})\hookrightarrow C_0(X_{\uu})$, 
which is the homogeneous basis element of $C_0(X_{\uu})$ corresponding to the vertex $v_*$. 
Sometime, we omit the upper index by writing $v_*=v_*^{f(v_*)}$ for brevity.
With these conventions, we can see that for each $\uu\in\Int^2$, the vector space $C_0(X_{\uu})$ is generated by all $v_*^{\uu}=t^{\uu-\gr(v_*)}v_*^{\gr(v_*)}$ such that $v_*$ is born before or at $\uu$, which means $\gr(v_*)\leq \uu$.
In fact, $C_0(X)$ is a free module and 
$\{v_b^{\gr(v_b)}, v_r^{\gr(v_r)}, v_g^{\gr(v_g)}\}$ forms a basis of $C_0(X)$.
That means, any element of $C_0(X)$ can be written as a $R$-linear combination of these $v_*^{\gr(v_*)}$'s and all these $v_*^{\gr(v_*)}$'s are linearly independent.

Similarly, for each $e_*\in\{e_b, e_r, e_g\}$,
we have the earliest basic element $e_*^{\gr(e_*)}$ of $C_1(X_{\gr(e_*)})$ for $\gr(e_*)\triangleq f(e_*)$. 
The set $\{e_b^{\gr(e_b}, e_r^{\gr(e_r)}, e_g^{\gr(e_g)}\}$ forms a basis of the free module $C_1(X)$. 
Furthermore, by the commutative property of morphisms, we have for each $\uu\geq\gr(e_*)$,
\begin{equation}
    \partial_{1,\uu}(e_*^{\uu})=
    \partial_{1,\uu}(\ttt^{\uu-\gr(e_*)} e_*^{\gr(e_*)})=
    \ttt^{\uu-\gr(e_*)}\circ\partial_{1,\gr(e_*)}(e_*)=
    \ttt^{\uu-\gr(e_*)}\circ\partial_1(e_*^{\gr(e_*)})
\end{equation}
In fact, as a morphism between two free modules, $\partial_{1}$ is fully determined by $\partial_1{(e_*)}$. Consider, for example, the red edge $e_r$ connecting $v_b$ and $v_r$. With the field chosen to be $\mathbb{F}_2$, one has $\partial_1(e_r^{\gr(e_r)})=\partial_1(e_r^{(1,1)})=v_b^{(1,1)}+v_r^{(1,1)}=\ttt^{(1,0)} v_b^{(0,1)}+\ttt^{(0,1)}v_r^{(1,0)}$. Similar for $\partial_1(e_b^{(1,2)})$ and $\partial_1(e_g^{(2,1)})$.
Therefore, $\partial_{1}$ can be represented as a matrix with entries in $R=\field{k}[\ttt]$ as follows:
\[
    \bordermatrix{[\partial_1]  &e_r^{(1,1)}            &   e_b^{(1,2)}         & e_g^{(2,1)}          \cr
    v_b^{(0,1)} & \mathbf{t}^{(1,0)}    &   \mathbf{t}^{(1,1)}   & 0                    \cr
    v_r^{(1,0)} & \mathbf{t}^{(0,1)}    &   0                   & \mathbf{t}^{(1,1)}    \cr
    v_g^{(1,1)} & 0                     &   \mathbf{t}^{(0,1)}   & \mathbf{t}^{(1,0)}   }
\]

\begin{figure}
    \centering
    \includegraphics[width=0.5\textwidth, page=2]{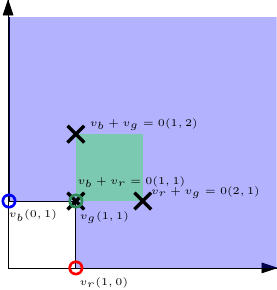}
    \caption{Persistence module whose presentation matrix is $[\partial_1]$ described
    in the working example. 
    }\label{fig:genrel}
\end{figure}

Now consider the $0^{th}$ persistence homology module induced from boundary morphism $\partial_1:C_1(X)\rightarrow C_0(X)$. Note that the $0^{th}$ homology is a space of connected components. For each $\uu$, $H_0(X_\uu)=\frac{Z_0(X_\uu)}{B_0(X_\uu)}=\frac{C_0}{\Img\partial_{1, \uu}}$. With bases of $C_0$ and $C_1$ chosen above, we have the $0^{th}$ persistence homology on grades from $(0,0)$ (bottom-left corner) to $(2,2)$ (top-right corner) described as following diagram (also
illustrated in Figure~\ref{fig:genrel}):

\[
\begin{tikzcd}
    \field{k} \arrow[rr, "1" description]                                        &  & \field{k} \arrow[rr, "1" description]                                        &  & \field{k}                             \\
                                                                         &  &                                                                      &  &                               \\
    \field{k} \arrow[uu, "1" description] \arrow[rr, "{[1,0]^\top}" description] &  & \field{k}^2 \arrow[uu, "{[1,1]}"] \arrow[rr, "{[1,1]}" description]          &  & \field{k} \arrow[uu, "1" description] \\
                                                                         &  &                                                                      &  &                               \\
    0 \arrow[rr, "0" description] \arrow[uu, "0" description]                                              &  & \field{k} \arrow[rr, "1" description] \arrow[uu, "{[1,0]^\top}" description] &  & \field{k} \arrow[uu, "1" description]
\end{tikzcd}
\]

\end{example}


In what follows, we take the liberty of omitting $X$ and $p$ if they are clear from the context. Thus, we may denote $Z_p(X)$, $B_p(X)$, and $H_p(X)$ as $Z$, $B$ and $H$ respectively.

\begin{definition}[decomposition]
For a finitely generated module $M$, we call $M\simeq \bigoplus M^i$ a  \emph{decomposition} of $M$ for some collection of modules $\{M^i\}$.
We say a module $M$ is \emph{indecomposable} if $M\simeq M^1\oplus M^2 \implies M^1=0$ or $M^2=0$.
By the Krull-Schmidt theorem~\cite{Atiyah1956},
there exists an essentially unique (up to permutation and isomorphism) decomposition $M\simeq \bigoplus M^i$ with every $M^i$ being indecomposable. We call it the \emph{total} decomposition of $M$.
\end{definition}
For example, the free module $R$ is generated by $<e_1^{(0,0)}>$ and the free module $R_{\rightarrow({0,1})}\oplus R_{\rightarrow{(1,0)}}$ is generated by $<e_1^{(0,1)}, e_2^{(1,0)}>$. A free module $M$ generated by $<e_j^{\mathbf{u}_j}:j=1,2,\cdots>$ has a (total) decomposition $M \simeq \bigoplus_{j}R_{\rightarrow \mathbf{u}_j}$.

\begin{definition}
Two morphisms $f:M\rightarrow N$ and $f':M'\rightarrow N'$ are isomorphic, denoted as $f\simeq f'$, if there exist isomorphisms $g:M\rightarrow M'$ and $h:N\rightarrow N'$ such that the following diagram commutes:
\begin{center}
\begin{tikzcd}
M \arrow[r, "f"] \arrow[d, " \overset{g}{\simeq}"' ] & N \arrow[d, "\overset{h}{\simeq}" ] \\
M' \arrow[r, "f'"]                                   & N'
\end{tikzcd}
\end{center}
\end{definition}
Essentially, like isomorphic modules, two isomorphic morphisms can be considered the same.
For two morphisms $f_1:M^1\rightarrow N^1$ and $f_2:M^2\rightarrow N^2$, there exists a canonical morphism $g:M^1\oplus M^2 \rightarrow N^1\oplus N^2, g(m_1, m_2)=(f_1 (m_1),\; f_2 (m_2))$, which is essentially uniquely determined by $f_1$ and $f_2$ and is denoted as $f_1\oplus f_2$.
We denote a trivial module by bold $\mathbf{0}$, and a trivial morphism by $\mathbb{0}$.
Analogous to the decomposition of a module, we can also define a decomposition of a morphism.



\begin{definition}
A morphism $f$ is indecomposable if $f\simeq f_1\oplus f_2 \implies f_1$ or $f_2$ is the trivial morphism $\mathbb{0}:\mathbf{0}\rightarrow \mathbf{0}$.
We call $f\simeq \bigoplus f_i$ a decomposition of $f$. If each $f_i$ is indecomposable, we call it a \emph{total} decomposition of $f$.
\end{definition}
Like decompositions of modules, the total decompositions of a morphism is also essentially unique.

\section{Presentation and its decomposition}\label{sec:presentation}
To study total decompositions of persistence modules as graded modules, we borrow the idea of \emph{presentations} of graded modules and build a bridge between decompositions of persistence modules and corresponding presentations. The later ones can be transformed to
a matrix reduction problem with possibly nontrivial constrains which we will introduce in Section~\ref{sec:computing_decomposition}. Our first main result is that there are 1-1 correspondences between persistence modules, presentations, and presentation matrices.
Recall that, by assumption, all modules are finitely generated.
A graded module hence a persistence module accommodates a description called its \emph{presentation} that aids finding its decomposition.
\begin{definition}
A presentation of a graded module $H$ is an exact sequence
\begin{tikzcd}[row sep=small, column sep = small]
F^1 \arrow[r, "f"] & F^0 \arrow[r, two heads] & H \arrow[r] & 0.
\end{tikzcd}
We call $f$ a presentation map. We say a graded module $ H $ is \emph{finitely presented} if there exists a presentation of $H$ with both $F^1$ and $F^0$ being finitely generated.
\end{definition}

It follows from the definition that a presentation of $H$ is determined by the presentation map $f$ where $\cok f \simeq H$.

\begin{remark}\label{rmk:presentation_and_minimal}
    Presentations of a given graded module are not unique. However, there exists an essentially unique (up to isomorphism) presentation $f$ of a graded module in the sense that any presentation $f'$ of that module can be written as $f'\simeq f\oplus f''$ with $\cok f''=0$. We call this unique presentation \emph{the minimal presentation}. See more details of the construction and properties of minimal presentation in Appendix~\ref{construction:appendix:free_resoln}.
\end{remark}

Fixed bases of nonzero free modules $F^1$ and $F^0$ provide a matrix form $[f]$ of the presentation map $f$, which we call a {\em presentation matrix} of $H$. It has entries in $R$.
In the special case that $H$ is a free module with $F^1$ being a zero module, we define the presentation matrix $[f]$ of $H$ to be a \emph{null column matrix} with matrix size $\ell\times 0$ for some $\ell\in \N$.
An important property of a persistence module $H$ is that a decomposition of its presentation $f$ corresponds to
a decomposition of $H$ itself.  The decomposition of
$f$ can be computed by
\emph{diagonalizing} its presentation matrix
$[f]$.
Informally, a diagonalization of a matrix $\mathbf{A}$ is an equivalent matrix $\mathbf{A}'$ in the following form (see formal Definition~\ref{def:diagonalization} later):


\[
\mathbf{A}'=
\left[
\begin{array}{c c c c}
\mathbf{A}_1 & 0 &\cdots & 0 \\
0 & \mathbf{A}_2 &  \cdots & 0\\
\vdots & \vdots & \ddots & \vdots \\
0 & 0 &\cdots & \mathbf{A}_k
\end{array}
\right]
\]

All nonzero entries are in $\mathbf{A}_i$'s and we write $\mathbf{A}\simeq \bigoplus \mathbf{A}_i$. It is not hard to see that for a map $f\simeq \bigoplus f_i$, there is a corresponding diagonalization $[f]\simeq \bigoplus [f_i]$.
With these definitions and
the fact that persistence modules are
graded modules, we have the following theorem that motivates our decomposition algorithm (proof in Appendix~\ref{sec:free_resoln_graded_betti}).


\begin{theorem} \label{thm:decomposition_correspondence}
There are 1-1 correspondences between the following three structures arising from a minimal presentation map $f: F^1\rightarrow F^0$ of a persistence  module $H$, and its presentation matrix $[f]$:
\begin{enumerate}
    \item A decomposition of the persistence module $H\simeq\bigoplus H^i$;
    \item A decomposition of the presentation map $f\simeq \bigoplus f_{i}$
    \item A diagonalization of the presentation matrix $[f]\simeq \bigoplus [f]_{i}$
\end{enumerate}

\end{theorem}


    

\begin{remark}\label{rmk:total_decomposition_correspondence}
    In practice, we might be given a presentation which is not necessarily minimal. One way to handle this case is to compute the minimal presentation of the given presentation first. For 2-parameter modules, this can be done by the algorithm in~\cite{Lesnick_compute_minimal_present_19}. The other choice is to compute the decomposition of the given presentation (not necessarily minimal) directly, which is 
    sufficient to get the decomposition of the module thanks to the following proposition 
    (proof at the end of Appendix~\ref{sec:free_resoln_graded_betti}).
    \end{remark}
    
    \begin{proposition}\label{prop:decomposition_correspondence}
    Let $f$ be any presentation of a graded module $H$.
    \begin{enumerate}
        \item For a decomposition of $H\simeq \bigoplus H^i$, there exists a decomposition of $f\simeq \oplus f^i$ so that $\cok f^i=H^i, \forall i $. 
        
        \item The total decomposition of $H$ follows from the total decomposition of $f$.
    \end{enumerate}
    \end{proposition}
    
    \begin{remark}\label{rmk:decompostion_with_trivial}
    
            By Remark~\ref{rmk:presentation_and_minimal}, any presentation $f$ can be written as $f\simeq f^*\oplus f'$ with $f^*$ being the minimal presentation and $\cok f'=0$. Furthermore, $f'$ can be written as $f'\simeq g\oplus h$  where $g$ is an identity map and $h$ is a zero map. The corresponding matrix form is $[f']\simeq [f^*]\oplus [g] \oplus [h]$ with $[g]$ being an identity submatrix and $[h]$ being a collection of zero column vectors. Therefore, one can easily read these trivial parts from the result of matrix diagonalization if it is total, meaning that none of its constituents ($\A_i$s) can be decomposed (diagonalized) further (Definition~\ref{def:diagonalization}).
           See the following diagram for an illustration. 
    
        \renewcommand{\kbldelim}{(}
        \renewcommand{\kbrdelim}{)}
    \[
        [f]=\kbordermatrix{
         &       & f^*     &     \vrule & & g & & \vrule    &h &   \\ \cline{2-10}   
        &       &         &     \vrule & &  & &  \vrule    &  &  \\
        &       &         &     \vrule & &  & &  \vrule    &  &  \\ 
        &       &[f^*]    &     \vrule & &  & &  \vrule    &  &  \\ 
        &       &         &     \vrule & &  & &  \vrule    &  &  \\
        &       &         &     \vrule & &  & &  \vrule    &  &  \\ \cline{2-7}    
        &       &         &     \vrule & 1&  & & \vrule &  &  \\
        &       &         &     \vrule & & 1 & & \vrule &  &  \\
        &       &         &     \vrule & &  & 1& \vrule  &  &  
     }
    \]
    \end{remark}

It follows from Theorem~\ref{thm:decomposition_correspondence} that we have to address the problem of 
\emph{total diagonalization} 
of a presentation matrix $[f]$, $f: F^1\rightarrow F^0$
for computing a total decomposition of the module
$H$ it represents. Each row $r_i$ and column $c_j$ of $[f]$ represent a generator $g_i$ and
a relation $s_j$ respectively where $[f]_{ij}=\alpha_{ij}$ if $s_j=\sum_i \alpha_{ij}g_i$ for $\alpha_{ij}\in R$. We label the rows and columns with the grades of the elements they represent, that is, $\gr(r_i):=\gr(g_i)$ and
$\gr(c_j):=\gr(s_j)$. Furthermore, we can simplify $[f]$ by observing that $\alpha_{ij}$
has the form $\alpha_{ij}= k\cdot t_1^{u_1}t_2^{u_2}$ where $k\in \field{k}$ and
${\bf u}=(u_1,u_2)=\gr(c_j)-\gr(c_i)$. For all
homogeneous transformations of bases, only
the value of $k$ changes in $\alpha_{ij}$. Therefore, we can replace $\alpha_{ij}$ with the
value of $k$ which is either $0$ or $1$ when
$\field{k}=\mathbb{F}_2$. 
See the matrices in Example~\ref{eg:eg0_matrix} given in Section~\ref{sec:simplification}.
With this change, the
matrix $[f]$ becomes a matrix over the field $\mathbb{F}_2$ with the following operations for transformations as summarized below:
\begin{enumerate}
    \item $[f]_{ij}=1$ if and only if $\alpha_{ij}=1$ in the relation
    $s_j=\sum_{i}\alpha_{ij}g_i$.
    \item A column $c_i$ can be added to column $c_j$, denoted as $c_i\rightarrow c_j$, only if $i\neq j \textrm{ and } \gr(c_i)\leq \gr(c_j)$. 
    A row $r_i$ can be added to row $r_j$ denoted $r_i\rightarrow r_j$ only if $i\neq j \textrm{ and } \gr(r_j)\leq \gr(r_i)$.
\end{enumerate}

\section{Computing decomposition}\label{sec:computing_decomposition}
In this section, we present an algorithm for computing a total decomposition of a distinctly graded module, which means that no two columns and no two rows in the presentation matrix have the same grades.
All modules are assumed to be finitely presented and
we take $\field{k}=\field{F}_2$ for simplicity though our method works for any finite field.
We have observed that a total decomposition of a module can be achieved by computing a total decomposition of its presentation $f$. This in turn means a total diagonalization of the presentation matrix $[f]$. Here we formally define some notations about the diagonalization.

Given an $\ell\times m$ matrix $\mathbf{A}=[{\mathbf{A}}_{i,j}]$, with row indices $\row(\mathbf{A})=[\ell]:=\{1,2,\cdots, \ell\}$ and column indices $\col({\mathbf{A}})=[m]:=\{1,2,\cdots, m\}$, 
we define an \emph{index block} $B$ of $\A$ as a pair 
$\big[\row(B),\, \col(B)\big]$ 
with $\row(B)\subseteq\row(\A), \col(B)\subseteq \col(\A)$. 
We say an index pair $(i,j)$ is in $B$ if $i\in \row(B)$ and $j\in \col(B)$, denoted as $(i,j)\in B$. 
We denote a \emph{block} of ${\mathbf{A}}$ on $B$ as the matrix restricted to the index block $B$, i.e. $[{\mathbf{A}}_{i,j}]_{(i,j)\in B}$, denoted as ${\mathbf{A}}|_B$. 
We call $B$ the index of the block ${\mathbf{A}}|_B$. We abuse the notations $\row(\A|_B):=\row(B)$ and $\col(\A|_B):=\col(B)$. For example, the $i$th row 
$r_i={\mathbf{A}}_{i,*}={\mathbf{A}}|_{[\{i\}, \col({\mathbf{A}})]}$ and the $j$th column 
$c_j={\mathbf{A}}_{*,j}={\mathbf{A}}|_{[\row({\mathbf{A}}), \{j\}]}$ are blocks with indices ${\big[\{i\}, \col({\mathbf{A}})]}$ and ${\big[\row({\mathbf{A}}), \{j\}]}$ respectively. Specifically, $\big[\emptyset, \{j\}\big]$ represents an index block of a single column $j$ and $[\{i\}, \emptyset]$ represents an index block of a single row $i$. We call $[\emptyset, \emptyset]$ the empty index block.


A morphism can have presentation matrices in different equivalent forms depending on the bases chosen. 

\begin{definition}
We say a matrix $\A'$ is equivalent to $\A$, denoted as $\A'\sim \A$, if they present the same morphism.  
\end{definition}


\begin{definition}\label{def:diagonalization}
A matrix ${\mathbf{A}}'\sim \mathbf{A}$ is called a \emph{diagonalization of $\A$} with a set of nonempty index blocks $\B=\{B_1, B_2, \cdots, B_k \}$
if rows and columns of $\A$ are partitioned into these blocks, i.e., 
$\row({\mathbf{A}})=\coprod_i \row(B_i)$ and $\col(\mathbf{A})=\coprod_i \col(B_i)$, and all the nonzero entries of $\A'$ have indices in some $B_i$.
We write $\mathbf{A}'= \bigoplus_{B_i\in \B} \A'|_{B_i}$. We say $\mathbf{A}'= \bigoplus_{B_i\in \B} \A'|_{B_i}$ is \emph{total} if no block in this diagonalization can be decomposed further into smaller nonempty blocks. That means, for each block $\A'|_{B_i}$, there is no nontrivial diagonalization. Specifically, when $\A$ is a null column matrix (the presentation matrix of a free module), we say $\A$ is itself a total diagonalization with index blocks $\{[\{i\}, \emptyset]\mid i\in \row(\A)\}$.
\end{definition}

Note that each nonempty matrix $\mathbf{A}$ has a trivial diagonalization with the set of index blocks being $\{(\row(\A), \col(\A))\}$. 
Guaranteed by Krull-Schmidt theorem~\cite{Atiyah1956}, 
all total diagonalizations are unique up to permutations of their rows and columns, and equivalent transformation within each block. The total diagonalization of $\A$ is denoted generically as ${\mathbf{A}}^*$. All total diagonalizaitons of $\mathbf{A}$ have the same set of index blocks, 
denoted as $\B^*$, unique up to permutations of rows and columns.



\subsection{Simplification of presentation matrix}
\label{sec:simplification}
First we want to transform the diagonalization problem to an equivalent problem that involves matrices with a simpler form. The idea is to simplify the presentation matrix to have entries only in $\field{k}$.
There is a correspondence between diagonalizations of the original presentation matrix and certain constrained diagonalizations of the corresponding transformed $\field{k}$-matrix under this subset of basic operations.

Inspired by the ideas from \cite{Carlsson2009}, we first make some observations about the homogeneous property of presentation maps and presentation matrices.
Equivalent matrices actually represent
isomorphic presentations $f'\simeq f$ that admit commutative diagram,
\[
\begin{tikzcd}
F^1 \arrow[r, "f"] \arrow[d, " \overset{h^1}{\simeq}"'] & F^0 \arrow[d, "\overset{h^0}{\simeq}"] \\
F^1 \arrow[r, "f'"] & F^0
\end{tikzcd}
\]
where $h^1$ and $h^0$ are endomorphisms on $F^1$ and $F^0$ respectively. The endomorphisms are realized by basis changes between corresponding presentation matrices $[f]\simeq [f']$.
Since all morphisms between graded modules are required to be homogeneous by definition, we can use homogeneous bases (all the basis elements chosen are homogeneous elements\footnote{Recall that an element $m\in M$ is homogeneous with grade $\gr(m)=\mathbf{u}$ for some $\mathbf{u}\in \Int^d$ if $m\in M_\mathbf{u}$.}) for $F^0$ and $F^1$ to represent matrices, say $F^0=<g_1, \cdots, g_n>$ and $F^1=<s_1, \cdots, s_m>$.
Therefore, for simplicity we can consider only equivalent presentation matrices under homogeneous basis changes. Each entry $[f]_{i,j}$ is also homogeneous. That means, $[f]_{i,j}=t^{\mathbf{u}}$ with $\uu=\gr(s_j)-\gr(g_i)$. We call such presentation matrix \emph{homogeneous presentation matrix}.

For example, given $F^0=<g_1^{(1,1)}, g_2^{(2,2)}>$,
the basis change $g_2^{(2,2)}\leftarrow g_2^{(2,2)}+g_1^{(1,1)}$ is not homogeneous since $g_2^{(2,2)}+g_1^{(1,1)}$ is no longer a homogeneous element.
However, $g_2^{(2,2)}\leftarrow g_2^{(2,2)}+\mathbf{t}^{(1,1)} g_1^{(1,1)}$ is a homogeneous change with $\gr(g_2^{(2,2)}+\mathbf{t}^{(1,1)} g_1^{(1,1)})=\gr(g_2^{(2,2)})= (2,2)$, which results in a new homogeneous basis,
$\{g_1^{(1,1)}, g_2^{(2,2)}+\mathbf{t}^{(1,1)} g_1^{(1,1)}\}$.
Homogeneous basis changes always result in homogeneous bases.

\vspace{0.1in}
\noindent
\textbf{Notation.}
Let $[f]$ be a homogeneous presentation matrix of $f:F^1\rightarrow F^0$ with bases $F^0=<g_1, \cdots, g_n>$ and $F^1=<s_1, \cdots, s_m>$.
We extend the notation of grading to every row $r_i$ and every column $c_j$ from the basis elements $g_i$ and $s_j$ they represent respectively, that is, $\gr(r_i):=\gr(g_i)$ and $\gr(c_j):=\gr(s_j)$. We define a partial order $\leq_{\gr}$ on rows $\{r_i\}$ by asserting $r_i\leq_{\gr} r_j$ iff $\gr(r_i)\leq \gr(r_j)$. Similarly, we define a partial order on columns $\{c_j\}$.

For such a homogeneous presentation matrix $[f]$, we have the following observations:
\begin{enumerate}
    \item  $\gr([f]_{i,j})=\gr(c_j)-\gr(r_i)$
    \item  a nonzero entry $[f]_{i,j}$ can only be zeroed out by  column operations from columns $c_k\leq_{\gr} c_j$ or by row operations from rows $r_\ell\geq_\gr r_i$.
\end{enumerate}


Observation (2) indicates which subset of column and row operations is sufficient to zero out the entry $[f]_{i,j}$. We restate the diagonalization problem as follows:

Given an $n\times m$ homogeneous presentation matrix ${\mathbf{A}}=[f]$ consisting of entries in $\field{k}[t_1,\cdots, t_d]$ with grading on rows and columns, find a total diagonalization of $\mathbf{A}$ under the following admissible row and column operations:

\begin{itemize}
    \item multiply a row or column by nonzero $\alpha \in \field{k}$; (For $\field{k}=\field{F}_2$, we can ignore these operations).
    \item for two rows $r_i, r_j$ with $j\neq i$ and $r_j\leq_{\gr} r_i$, set $r_j\leftarrow r_j+\mathbf{t}^{\mathbf{u}}\cdot r_i$ where $\mathbf{u}=\gr(r_i)-\gr(r_j)$
    \item for two columns $c_i, c_j$ with $j\neq i$ and $c_i \leq_{\gr} c_j$, set $c_j\leftarrow c_j+\mathbf{t}^{\mathbf{v}}\cdot c_i$ where $\mathbf{v}=\gr(c_j)-\gr(c_i)$ 
\end{itemize}

The above operations realize all possible homogeneous basis changes. That means, any homogeneous presentation matrix can be realized by a combination of the above operations.

In fact, the values of nonzero entries in the matrix are redundant under the homogeneous property $\gr(\A_{i,j})=\gr(c_j)-\gr(r_i)$ given by observation (1). 
So, we can further simplify the matrix by replacing all the nonzero entries with 
their $\field{k}$-coefficients. For example, we can replace $2\dotr\mathbf{t}^\uu$ with $2$. What really matters are the partial orders defined by the grading of rows and columns.
With our assumption of $\field{k}=\field{F}_2$, all nonzero entries are replaced with $1$.
Based on above observations, we further simplify the diagonalization problem to be the one as follows.

Given a $\field{k}$-valued matrix $\mathbf{A}$ with a partial order on rows and columns, find a total diagonalization $\mathbf{A}^*\sim \mathbf{A}$ with the following admissible operations:
\begin{itemize}
    \item multiply a row or column by nonzero $\alpha \in \field{k}$; (For $\field{k}=\field{F}_2$, we can ignore these operations).
    \item Add $c_i$ to $c_j$ only if $j\neq i \mbox{ and } \gr(c_i) \leq \gr(c_j)$; denoted as $c_i\rightarrow c_j$.
    \item Add $r_k$ to $r_l$ only if $l\neq k \mbox{ and } \gr(r_\ell)\leq \gr(r_k)$; denoted as $r_k\rightarrow r_l$.
\end{itemize}

The assumption of $\field{k}=\field{F}_2$ allows us to 
ignore the first set of multiplication operations on the binary matrix obtained after transformation. Also, with the assumption of distinct grading, the second two sets of admissible operations become:
\begin{itemize}
    \item Add column $c_i$ to column $c_j$, denoted as $c_i\rightarrow c_j$, only if $\gr(c_i) < \gr(c_j)$. 
    \item Add row $r_i$ to row $r_j$, denoted $r_i\rightarrow r_j$, only if $\gr(r_i) > \gr(r_j)$.
\end{itemize}

We denote the set of all admissible column and row operations as
\begin{align*}
  \colop=&\{(i,j)\mid c_i\rightarrow c_j \mbox{ is an admissible column operation}\},\\
  \rowop=&\{(k,l)\mid r_k\rightarrow r_l \textrm{ is an admissible row operation}\}.
\end{align*}
Under the assumption that no two columns nor rows have same grades,
$\colop$ and $\rowop$ are closed under transitive relation.

\begin{proposition}\label{prop:transitive relation}
$(i,j), (j,k)\in \colop \ (\rowop) \implies (i,k)\in \colop \ (\rowop)$.
\end{proposition}


Given a solution of the diagonalization problem in the simplified form, one can reconstruct a solution of the original problem on the presentation matrix by reversing the above process of simplification. We will illustrate it by running our algorithm on the working example~\ref{ex:working_example} at the end of this section.
The matrix reduction we employ for diagonalization may be viewed as a \emph{generalized matrix reduction} because the matrix is reduced under constrained operations $\colop$ and $\rowop$ which might be a nontrivial subset of all basic operations.
\begin{remark}
There are two extreme but trivial cases: (i) there are no $\leq_\gr$-comparable pair of rows and columns. In this case, $\colop=\rowop=\emptyset$ and the original matrix is a trivial solution. (ii) All pairs of rows and all pairs of columns are $\leq_\gr$-comparable. Or equivalently, both $\colop$ and $\rowop$ are totally ordered. In this case,
one can apply traditional matrix reduction algorithm to reduce the matrix to a diagonal matrix with all nonzero blocks being $1\times 1$ minors.
This is also the case for traditional 1-parameter persistence module if one further applies row reduction after column reduction. Note that row reductions are not necessary for reading out persistence information because it essentially does not change the persistence information. However, in multiparameter cases, both column and row reductions are necessary to obtain a diagonalization from which the persistence information can be read. From this view-point, our algorithm can be thought of as a generalization of the traditional persistence algorithm.
\end{remark}

\begin{example}\label{eg:eg0_matrix}
Consider our working example~\ref{ex:working_example}. One can see later in Section~\ref{sec:compte_presentation} (Case 1) that the presentation matrix of this example can be chosen to be the same as the matrix of the boundary morphism $\partial_1: C_1\rightarrow C_0$. With fixed bases
$C_0 = <v_b^{(0,1)}, v_r^{(1,0)}, v_g^{(1,1)}>$ and $C_1=<e_r^{(1,1)}, e_b^{(1,2)}, e_g^{(2,1)}>$,
this presentation matrix $[\partial_1]$ and the corresponding binary matrix $\A$ can be written as follows (recall that superscripts indicate the grades) :
\[
    \bordermatrix{[\partial_1]  &e_r^{(1,1)}            &   e_b^{(1,2)}         & e_g^{(2,1)}          \cr
    v_b^{(0,1)} & \mathbf{t}^{(1,0)}    &   \mathbf{t}^{(1,1)}   & 0                    \cr
    v_r^{(1,0)} & \mathbf{t}^{(0,1)}    &   0                   & \mathbf{t}^{(1,1)}    \cr
    v_g^{(1,1)} & 0                     &   \mathbf{t}^{(0,1)}   & \mathbf{t}^{(1,0)}   }
    \longrightarrow
    \bordermatrix{\mathbf{A}  &c_1^{(1,1)}            &   c_2^{(1,2)}         &  c_3^{(2 ,1)}          \cr
    r_1^{(0,1)} & 1    &   1 & 0                    \cr
    r_2^{(1,0)} & 1    &   0                   & 1    \cr
    r_3^{(1,1)} & 0                     &   1   & 1   }
\]
Four admissible operations are:
$r_3\rightarrow r_1, r_3\rightarrow r_2, c_1\rightarrow c_2, c_1\rightarrow c_3$.
\end{example}


\subsection{Total diagonalization algorithm}

Recall that we assume distinct grading, i.e., no two columns nor two rows have same grades.
We make some comments on the output of our algorithm without this assumption later in the conclusion.


Let $\mathbf{A}$ be the presentation matrix whose total diagonalization we are looking for.
We order the rows and columns of the matrix $\mathbf{A}$ according to any topological order that extends the partial  order on the grades to a total order, e.g., dictionary order. We fix the indices $\row(\A)=\{1,2,\cdots, \ell\}$ and $\col(\A)=\{1,2,\cdots, m\}$ according to this order. With this ordering, observe that, for each admissible column operation $c_i\rightarrow c_j$, we have $i< j$, and for each admissible row operation $r_l\rightarrow r_k$, we have $l> k$.


For any column $c_t$, let $\mathbf{A}_{\leq t}:=\mathbf{A}|_C$ denote the left submatrix on $C=\big[\row(\mathbf{A}), \{j\in \col(\mathbf{A})\mid j\leq t\}\big]$ and $\mathbf{A}_{<t}$ denote its stricter version obtained by excluding column $c_t$ from $\mathbf{A}_{\leq t}$. Our algorithm starts with the finest decomposition and coarsens it as necessary. The main steps of our algorithm runs as follows (see Figure~\ref{Fig:incremental1} for an illustration): 

\begin{figure}[!htb]
        \centering
         \includegraphics[width=\linewidth, page=3]{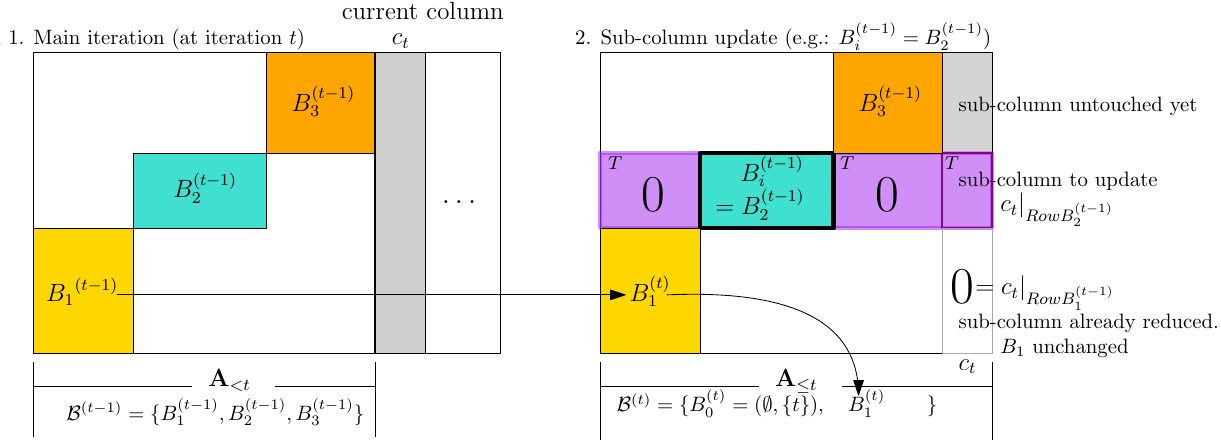}
         \caption{
            (left) $\A$ at the beginning of iteration $t$ with $\A_{<t}$ being totally diagonalized with three index blocks $\B^{(t-1)}=\{B_1^{(t-1)}, B_2^{(t-1)},B_3^{(t-1)}\}$.
            (right) A sub-column update step:  
            $c_t|_{\row B_1^{(t-1)}}$ has already been reduced to zero. So, $B_1^{(t)}=B_1^{(t-1)}$ is added into $\B^{(t)}$. White regions including $c_t|_{\row B_1^{(t-1)}}$ must be preserved afterward. 
            Now for $i=2$, we attempt to reduce purple sub-column $c_t|_{\row B_2^{(t-1)}}$.
            We extend it to block on ${T}:=\big[\row(B_2^{(t-1)}),\,  (\col(A_{\leq t})\setminus \col(B_2^{(t-1)}))\big]$ and try to reduce it in {\sc{BlockReduce}}.
         }\label{Fig:incremental1}
\end{figure}

\begin{enumerate}
\item[0.]{\bf Initialization}: 
Initialize the collection of index blocks $\mathcal{B}^{(0)}:=\{B_i^{(0)}:=\big[\{i\},\, \emptyset\big]\mid i\in \row(\A)\}$, for the total diagonalization of null column matrix $\A_{\leq 0}$.

\item[1.]{\bf Main iteration}: Process $\A$ from left to right incrementally by introducing a column $c_t$
and considering left submatrices $\mathbf{A}_{\leq t}$ for $t=1,2,\cdots, m$. We update and maintain the collection of index blocks $\B^{(t)}\leftarrow\{B_i^{(t)}\}$ for the current submatrix $\A_{\leq t}$ in each iteration by using column and block updates stated below. Here we use upper index $(\cdot)^{(t)}$ to emphasize the iteration $t$. 

\item[2.]{\bf Sub-column update}: 
Partition the column $c_t$ into sub-columns $c_t|_{\row B_i^{(t-1)}}:=\A_{[\row(B_i^{(t-1)}),\, \{t\}]}$, one for the set of rows $\row(B_i^{(t-1)})$ for each block from the previous iteration.
We process each such sub-column $c_t|_{\row B_i^{(t-1)}}$ one by one, checking whether there exists a sequence of admissible operations that are able to reduce the sub-column to zero while {\emph{preserving the prior}}, 
according to the definition below.
\begin{definition}
A \emph{prior} with respect to a sub-column $c_t|_{\row B_i^{(t-1)}}$ is defined to be the left submatrix $A_{<t}$ and sub-columns $c_t|_{\row B_j^{(t-1)}}$ for all $j<i$.
\end{definition}
That is to say, the operations that preserve prior, together \emph{change neither $\A_{<t}$ nor other sub-columns $c_t|_{\row B_j^{(t-1)}}$ for all $j<i$.} If such operations exist, we apply them on the current $\A$ to get an equivalent matrix with the sub-column $c_t|_{\row B_i^{(t-1)}}$ being zeroed out and we set $B_i^{(t)}\leftarrow B_i^{(t-1)}$. Otherwise, we leave the matrix $\A$ unchanged and add the column index $t$ to those of
$B_i^{(t-1)}$, i.e., we set $B_i^{(t)}\leftarrow \big[\row(B_i^{(t-1)}), \col(B_i^{(t-1)})\cup \{t\}\big]$. 
After processing every sub-column $c_t|_{\row B_i^{(t-1)}}$
one by one, all index blocks $B_i^{(t)}$ containing column index $t$ are merged into one single index block.
At the end of iteration $t$, we get an equivalent matrix $\A$ with $\mathbf{A}_{\leq t}$ being totally diagonalized with index blocks $\B^{(t)}$.

\item[3.]{\bf Block reduce}: To update the entries of each sub-column of $c_t$ described in 2, we propose a block reduction algorithm {\sc BlockReduce} to compute the correct entries. Given
$T:=\big[\row(B_i^{(t-1)}),\, (\col(\A_{\leq t})\setminus\col(B_i^{(t-1)}))\big]$,
this routine checks whether the block $T$
 can be zeroed out by some collection of admissible operations. If so, $c_t$ does not join the block $B_i^{(t)}$ and $\A$ is updated with these operations.
\end{enumerate}


For two index blocks $B_1,B_2$, we denote the merging $B_1\oplus B_2$ of these two index blocks as an index block $\big[\row({B_1})\cup \row({B_2}),\, \col({B_1})\cup \col({B_2})\big]$.
In the following algorithm, we treat the given matrix $\A$ to be a global variable which can be visited and modified anywhere by every subroutines called. That means, every time we update values in $\A$ by some operations, these operations are applied to the latest $\A$.

\begin{algorithm}[H]
 \caption{{\sc TotDiagonalize}($\mathbf{A}$)}
 \KwIn{$\mathbf{A}=$ input matrix treated as a global variable whose columns and rows are totally ordered respecting some fixed partial order given by the grading.}
\SetAlgoLined
\KwResult{a total diagonalization $\A^*$ with index blocks $\mathcal{B}^{*}$}

 $\B^{(0)}\leftarrow \{B_i^{(0)}:=\big[\{i\}, \emptyset\big] \mid {i\in \row(\mathbf{A})}\}$\; 

 \For{ $t\gets 1$ to $m:=|\col(\mathbf{A})|$
 }{
    $B_0^{(t)}\leftarrow\big[\emptyset, \{t\}\big]$\;
    \For{each $B_i^{(t-1)}\in \B^{(t-1)}$}{
        $T:=\big[\row({B_i^{(t-1)}}),\, \col(\mathbf{A}_{\leq t})\setminus\col({B_i^{(t-1)}})\big]$\;
        \If{{\sc BlockReduce} $(T)$== {\tt false}}{
            
            $B_i^{(t)}\leftarrow B_i^{(t-1)}\oplus B_0^{(t)}$; \tcp*[f]{update $B_i$ by appending $t$}\\
        }
        \Else{
            $B_i^{(t)}\leftarrow B_i^{(t-1)}$; \tcp*[f]{$B_i$ remains unchanged}\\
        }
    }
    $\B^{(t)}\leftarrow \{B_i^{(t)}\}$ with all $B_i^{(t)}$ containing $t$ merged as one block.\\
    \tcp{$\A$ and $c_t$ are updated in {\sc BlockReduce} whenever it returns {\tt true} }
}
\KwRet{$(\A, \B^{(m)})$\;}
\end{algorithm}

\begin{remark}
 Our algorithm does not require the input presentation matrix to be minimal. As indicated in Remark~\ref{rmk:decompostion_with_trivial}, the trivial parts result in either identity blocks or single column blocks like $\big[\emptyset, \{j\}\big]$. A single column block corresponds to a zero morphism and is not merged with any other blocks. Therefore, $c_j$ is a zero column. For a single row block $\big[\{i\},\, \emptyset\big]$ which is not merged with any other blocks, $r_i$ is a zero row vector. It represents a free indecomposable submodule in the total decomposition of the input persistence module.
\end{remark}

We first prove the correctness of {\sc TotDiagonalize} assuming that {\sc BlockReduce} routine works
as claimed, namely, it
checks if a sub-column of the current column $c_t$ can be zeroed out while preserving the prior,
that is, without changing the left submatrix from
the previous iteration and also the other sub-columns of $c_t$ that have already been zeroed out.
\begin{proposition}\label{prop:alg_correctness}
At the end of each iteration $t$, $\mathbf{A}_{\leq t}$ is a total diagonalization. 
\label{main-iter}
\end{proposition}
\begin{proof}

    We prove it by induction on $t$. For the base case $t=0$, it follows trivially by definition.
    Now assume $\A^{(t-1)}$ is the matrix we get at the end of iteration $(t-1)$ with $\A^{(t-1)}_{\leq t-1}$ totally diagonalized. That means, $\A^{(t-1)}_{\leq t-1}=\A_{\leq t-1}^*$ where $\A=\A^{(0)}$ is the original given matrix. 
    For contradiction, assume at the end of iteration $t$, the matrix we get, $\A^{(t)}$, has left submatrix $\A^{(t)}_{\leq t}$ which is not totally diagonalized. That means, some index block $B\in\B^{(t)}$ can be decomposed further. 
    Observe that such $B$ must contain $t$ because all other index blocks (not containing $t$)
    in $\B^{(t)}$ are also in $\B^{(t-1)}$ which cannot be decomposed further by our inductive assumption. We denote this index block containing $t$ as $B_t$. Let $\A'$ be the equivalent matrix of $\A^{(t)}$ such that $\A'_{\leq t}$ is a total diagonalization with index blocks $\B'$. 
    Let $F$ be an equivalent transformation from $\A^{(t)}$ to $\A'$, which decomposes $B_t$ into at least two distinct index blocks of $\B'$, say $B_0, B_1, \cdots$. Only one of them contains $t$, say $B_0$. Then $B_1$ consists of only indices that are from $\A_{\leq t-1}$, 
    which means $B_1$ equals some index block $B_i\in \B^{(t-1)}$. Therefore, the transformation $F$ gives a sequence of admissible operations which can reduce the sub-column $c_t|_{\row(B_i)}$ to zero in $\A^{(t)}$. 
    Note that we just use $F$ to decompose the block of $B_t$. Therefore, we can choose a sequence of admissible operations which only involves indices of $B_t$. This gives us a sequence of admissible operations that does not change other sub-columns $c_t|_{\row(B_j)}$ for $B_j\neq B_t$.
    Starting with this sequence of admissible operations, we construct another sequence of admissible operations which further keeps $\A^{(t)}_{\leq t-1}$ unchanged to reach the contradiction. Note that $\A^{(t)}_{\leq t-1}=\A^{(t-1)}_{\leq t-1}$

    Observe that all index blocks of $\B'$ other than $B_0$ are also index blocks in $\B^{(t-1)}$, i.e. $\B'\setminus\{B_0\}\subseteq \B^{(t-1)}$. $B_0$ can be written as $B_0=\bigoplus_{B_j \in \B^{(t-1)}\setminus\B'} B_{j}\oplus [\emptyset, \{t\}]$. Let $B_a$ be the merge of index blocks that  are in $\A^{(t-1)}$ and also in $\A'$ and $B_b$ be the merge of the rest of the index blocks of $\A^{(t-1)}$, i.e., $B_a = \bigoplus_{B_j\in \B'\cap \B^{(t-1)}} B_j$ and $B_b = \bigoplus_{B_j \in \B^{(t-1)}\setminus\B'} B_{j}$. Then $\{B_a, B_b\}$ can be viewed as a coarser decomposition on $\A^{(t-1)}_{\leq t-1}$ and also on $\A'_{\leq t-1}$. By taking restrictions, 
    we have $\A'|_{B_a}\sim \A^{(t-1)}|_{B_a}$ with equivalent transformation $F_a$ and $\A'|_{B_b}\sim \A^{(t-1)}|_{B_b}$ with equivalent transformation $F_b$. Then $F_a$ gives a sequence of admissible operations with indices in $B_a$ and $F_b$ gives a sequence of admissible operations with indices in $B_b$. By applying these operations on $\A'$, we can transform $\A'_{\leq t-1}$ to $\A^{(t-1)}_{\leq t-1}$ with sub-column $[\row(\A)\setminus\row(B_0), \{t\}]$ unchanged, which consists of the sub-columns that have already been reduced to zero. 
    Combining all admissible operations from the three transformations $F, F_a$ and $F_b$ together, we get a sequence of admissible operations that reduce sub-column $[\row(B_i),\, \{t\}]$ to zero without changing $\A^{(t)}_{< t}$ and also those sub-columns which have already been reduced. But, then
    {\sc BlockReduce} would have returned `true' signaling that $B_i$ should not be merged with any other block
    required to form the block $B_t$ reaching a contradiction. 
\end{proof}

Now we design the {\sc BlockReduce} subroutine as required. 
With the requirement of prior preservation,
observe that reducing the sub-column $c_t|_{\row B}$ for some $B\in\B^{(t-1)}$ is the same as reducing ${T}=[\row(B),\, (\col(\A_{\leq t})\setminus\col(B))]$ called the {\emph{target block}} (see Figure~\ref{Fig:incremental1} on the right) . The main idea of {\sc BlockReduce} is to consider a 
specific subset of admissible operations called \emph{independent operations}. 
Within $\A_{\leq t}$,
these operations only change entries in ${T}$ and this change is independent of their order of application. Our {\sc BlockReduce} is designed to search for a sequence of admissible operations within this subset and
reduce ${T}$ with it, if it exists. Clearly, the prior is preserved with these operations. 
The only thing we need to ensure is that searching within 
the set of independent operations is sufficient. That means, if there exists a sequence of admissible operations that can reduce ${T}$ to $0$ and meanwhile preserves the prior, then we can always find one such sequence with only independent operations. This is what we show next.


Consider the following matrices for each admissible operation.
For each admissible column operation $c_i\rightarrow c_j$, 
let
\[
 \Y^{i,j}:=\A\dotr[\delta_{i,j}]
\]
where $[\delta_{i,j}]$ is the $m\times m$ square matrix with only one non-zero entry at $(i,j)$. Observe that $\A\dotr[\delta_{i,j}]$ is a matrix with the only nonzero column at $j$ with entries copied from $c_i$ in $\A$. 
Similarly, for each admissible row operation $r_l\rightarrow r_k$, let $[\delta_{k,l}]$ be the $\ell\times\ell$ matrix with only non-zero entry at $(k,l)$.
let 
\[
\X^{k,l}:=[\delta_{k,l}]\dotr\A
\]

Application of a column operation $c_i\rightarrow c_j$ 
can be viewed as
updating $\A$ to $\A\dotr(\I+[\delta_{i,j}])=\A+\Y^{i,j}$.
Similar observation holds for row operations as well.
For a target block $T=[\row(B), \col(\A_{\leq t})\setminus\col(B)]$ defined on some $B\in\B^{(t-1)}$, we say an admissible column (row)  operation, $c_i\rightarrow c_j$ ($r_l\rightarrow r_k$ resp.) is \emph{independent} on $T$ if $i\notin \col(T), j\in col(T)$ ($l\notin \row(T), k\in \row(T)$ resp.). Briefly, we just call such operations \emph{independent operations} if $T$ is clear from the context.



We have two observations about independent operations that are important.
The first one follows from the definition that ${T}=[\row(B),\, \col(\A_{\leq t})\setminus\col(B)]$.
\begin{observation}
Within $\A_{\leq t}$, an independent column or row operation only changes entries on $T$.
\label{obs:prior}
\end{observation}
\begin{observation}\label{prop:cross_term_0}
For any independent column operation $c_i\rightarrow c_j$ and row operation $r_l\rightarrow r_k$,
we have $[\delta_{k,l}]\dotr {\A}\dotr [\delta_{i,j}]=0$. Or, equivalently
\begin{equation}\label{eq:cross_term_0}
(\I_{\ell\times\ell} + [\delta_{k,l}] )\dotr \A \dotr (\I_{m\times m}+[\delta_{i,j}])=\A + [\delta_{k,l}]\A + \A [\delta_{i,j}]= \A+\X^{k,l} + \Y^{i,j}
\end{equation}
\end{observation}
\begin{proof}
$[\delta_{k,l}]\dotr {\A}\dotr [\delta_{i,j}]=\A_{l,i}[\delta_{k,j}]$(see Fig~\ref{fig:cross_term} for an illustration). By definitions of independence and $T$, we have $l\notin\row(B), i\in \col(B)$. That means they are row index and column index from different blocks. Therefore, $\A_{l,i}=0$.
\end{proof}
\begin{figure}
    \centering
    \includegraphics[width=0.3\textwidth]{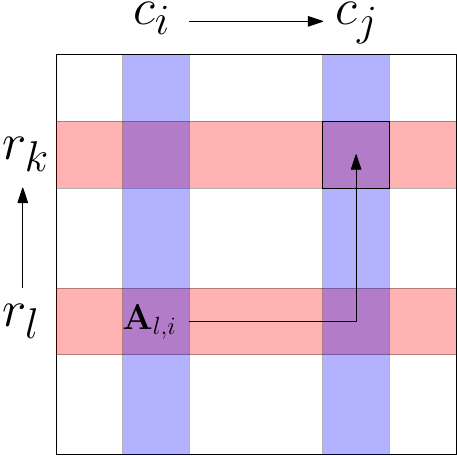}
    \caption{$[\delta_{k,l}] \A [\delta_{i,j}]$ is a matrix with the only nonzero entry at $(k,j)$ being a copy of $\A_{l,i}$.}
    \label{fig:cross_term}
\end{figure}

The following proposition reveals why we are after the independent
operations.

\begin{proposition}\label{prop:block_reduction_linearization}
The target block $\A|_{{T}}$ can be reduced to 0 while preserving the prior if and only if $\A|_{{T}}$ can be written as a linear combination of independent operations. That is,
\begin{equation}
    \A|_{{T}}=
\sum_{\mathclap{\substack{l\notin\row(T)\\ k\in\row({T}) }}}\alpha_{k,l} \X^{k,l}|_{{T}}+\sum_{\mathclap{\substack{i\notin\col(T)\\ j\in \col({T}) }}} \beta_{i,j}\Y^{i,j}|_{{T}} 
\label{eq:ops}
\end{equation}
where 
$\alpha_{k,l}$'s and $\beta_{i,j}$'s are coefficient in $\mathbb{k}=\mathbb{F}_2$. 
\end{proposition}
\begin{proof}

    The full proof is in Appendix~\ref{sec:missing_proofs}, here we give some intuitive explanation. 
    Reducing the target block $\A|_{{T}}$ to $0$ is equivalent to finding matrices $\LL$ and $\RR$
    encoding sequences of admissible row operations and admissible column operations respectively so that $\LL\A\RR|_{{T}}=0$. 
    For $\Leftarrow$ direction, 
    we can build 
    $\LL=\I+\sum\alpha_{k,l}[\delta_{k,l}]$ and 
    $\RR=\I+\sum\beta_{i,j}[\delta_{i,j}]$ 
    with binary coefficients $\alpha_{k,l}$'s and $\beta_{i,j}$'s given in Equation~\ref{eq:ops}. Then using
    Observations~\ref{obs:prior} and~\ref{prop:cross_term_0},  
    we show $\LL\A\RR$ indeed reduces $\A|_{{T}}$ to $0$ with the prior being preserved.
    
    For $\Rightarrow$ direction, as long as we show that the existence of a transformation reducing $\A|_{{T}}$ to $0$ implies the existence of a transformation reducing $\A|_{{T}}$ to $0$ by independent operations, we are done. This is formally captured as Proposition~\ref{lm:linear_comb} and proved in Appendix~\ref{sec:missing_proofs}. 
\end{proof}
\begin{figure}[h!]
    \centering
    \includegraphics[width=\textwidth, page=18]{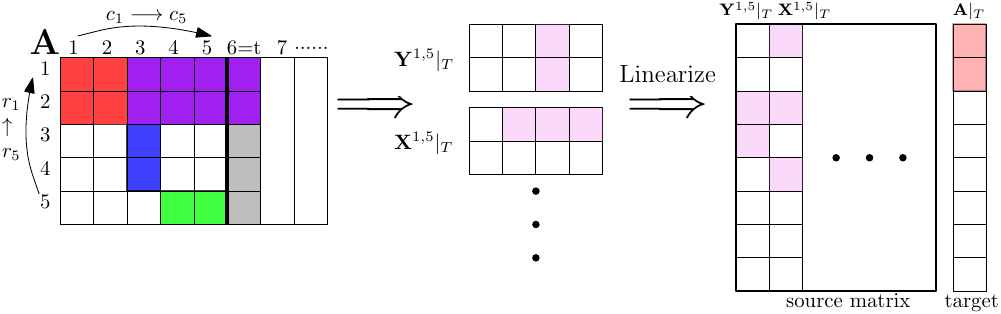} 
    \caption{$t=6$ is the current column. For the current block, the red one, our target block $T$ is the purple one. We build $\X^{k,l}|_T$'s for admissible row operations from blue or green into purple. For example, $r_5\rightarrow r_1$ illustrated on the left. Also, build $\Y^{i,j}|_T$'s for admissible column operations from the red block to the purple block. For example, $c_1\rightarrow c_5$ on top-left. The middle picture shows $\X^{1,5}|_{{T}}$ and  $\Y^{1,5}|_{{T}}$ for these two operations. After linearizing, the corresponding  vectors are added into the source matrix, which is finally used to reduce the target $\A|_{T}$.
     }
    \label{fig:linearization_trick_illustration}
\end{figure}
We can view $\A|_T, \Y^{i,j}|_T, \X^{k,l}|_T$ as binary vectors in the same $|T|$-dimensional space.
Proposition~\ref{prop:block_reduction_linearization}
tells us that it is sufficient to
check if $\A|_{{T}}$ can be a linear combination of the vectors
corresponding to a set of independent operations. So, we first
\emph{linearize} each of the matrices $\Y^{i,j}|_T$'s, $\X^{k,l}|_T$'s, and
$\A|_{T}$ to a column vector as described later (see Figure~\ref{fig:linearization_trick_illustration}). Then, we check
if $\A|_{{T}}$ is in the span of $\Y^{i,j}|_{{T}}$'s and
$\X^{k,l}|_{{T}}$'s. This is done by collecting all vectors $\X^{i,j}|_T$'s
and $\Y^{k,l}|_T$'s into a matrix $S$ called the \emph{source matrix} (Figure~\ref{fig:linearization_trick_illustration}(right)) and then
reducing the vector $c:=\A|_T$ with $S$ by some standard matrix reduction algorithm with left-to-right column additions, which is the subroutine called {\sc ColReduce} in {\sc BlockReduce} described below.
If $c=\A|_T$ can be reduced to
$0$, we apply the corresponding independent operations to update $\A$.
Observe that all column operations used in reducing $\A|_T$ together only change
the sub-column $c_t|_{\row B}$ while row operations 
may change $\A$ to the right of the column $t$.

Here we provide a short description and the pseudo-code of the subroutine {\sc ColRdeuce}.

For a column $c_j$, we use $\low(c_j)$ to indicate the lowest row number such that $c_j$ has $1$ in that row. Let $\low(c_j)=-1$ if $c_j$ is a zero column.
We call a matrix $\mathbf{S}'\sim \mathbf{S}$  lowest-conflict-free for $\mathbf{S}$ if for each row index $i=\low(c_j)\neq -1$ there is no $j'\not = j$ so that $\low(c_{j'})=i$.
Notice that $\mathbf{S}'$ is not necessarily unique. However, all the claims do not depend on the choice of $\mathbf{S}'$.
The algorithm {\sc ColReduce}$(\mathbf{S}, c)$ transforms the matrix $[\mathbf{S}|c]$ to a lowest-conflict-free matrix and as a result reduces the column $c$.
We say this procedure {\em reduces $c$ with $\mathbf{S}$.}
Note that this algorithm is the traditional persistence algorithm.
\begin{algorithm}[H]
    \caption{\sc{ColReduce}$(\mathbf{S
    }, c)$ }
    \SetAlgoLined
    \KwIn{$\mathbf{S}$=source matrix, $c$=target column to reduce.}
    \KwResult{return the reduced target column}
    $\BS'\gets [\BS|c]$\;
    \For(\tcp*[f]{Transform $[\mathbf{S}|c]$ to be lowest-conflict-free}){$i\gets 1$ to $|\col(\mathbf{S})|$}
    {
        $\ell\gets \low(c_i)$\;
        \If{$\ell\neq -1$}{
        \For{ $j\gets 1$ to $i-1$}{
            \If{ $\low(c_j)==\ell$}{
                $c_i\gets c_j+c_i$\; go to 3
            }
        }
        }
    }
    \KwRet{$c$}
\end{algorithm}

The following fact is well known and is the basis of the classical matrix based
persistence algorithm.

\begin{fact}\label{fact:col_red}
There exists a set of column operations adding a column only to its right
such that the matrix $[\mathbf{S}|c]$ is reduced to $[\mathbf{S}'|0]$
if and only if {\sc ColReduce}$(\mathbf{S},c)$ returns a zero vector.
\end{fact}

Now we describe the linearization used in routine {\sc BlockReduce} as presented in
Algorithm~\ref{alg:blockreduce}:{\sc BlockReduce}.
We fix a linear order $\leq_{\lin}$ on the set of matrix indices, 
$\row(\A)\times \col(\A) $, as follows:
$(i,j)\leq_{\lin}(i',j')$ if $j>j'$ or $j=j', i<i'$. 
Explicitly, we linearly order the indices as: 
\begin{equation*}
    (({1,m}), ({2,m}),\dots, ({\ell,n}),({1,m-1}),({2,m-1}),\dots).
\end{equation*}
For any index block $B$, let $\lin(\mathbf{A}|_B)$ be the vector of dimension
$|\col(B)|\cdot |\row(B)|$ 
obtained by linearizing $\mathbf{A}|_B$ 
to a vector in the above linear order on the indices.

\begin{algorithm}[H]
    \caption{{\sc BlockReduce}$(T)$}\label{alg:blockreduce}
    \SetAlgoLined
    \KwData{$\A$=global variable of the given matrix.}
    \KwIn{
    $T$=index of target block to be reduced;
    $t$=index of current column
    }
    \KwResult{Return a boolean to indicate whether $\A|_{T}$ can be reduced. Reduce block $\A|_{T}$ if possible.}
    Compute $c:=\lin(\mathbf{A}|_T)$ and initialize empty matrix $\mathbf{S}$\;
    \For{each admissible column operation $c_i\rightarrow c_j$ with $i\notin \col(T), j \in \col(T)$,}{
        compute $\Y^{i,j}|_T:=(\A\dotr[\delta_{i,j}])|_T$ and $y^{i,j}=\lin(\Y^{i,j}|_T)$; update $\mathbf{S}\leftarrow [\mathbf{S}|y^{i,j}]$\;
    }
    \For{each admissible row operation $r_l\rightarrow r_k$ with $l\notin \row(T), k\in \row(T)$}{
        compute $\X^{k,l}|_T:= ([\delta_{k,l}]\dotr\A)|_T$ and $x^{k,l}:=\lin(\X^{k,l}|_T)$; update $\mathbf{S}\leftarrow [\mathbf{S}|x^{k,l}]$\;
    }
    {\sc ColReduce} $(\mathbf{S},c)$ returns indices of $y^{i,j}$'s and $x^{k,l}$'s used to reduce $c$ if possible\;
    For every such index of $y^{i,j}$ or $x^{k,l}$ apply $c_i\rightarrow c_j$ or $r_l\rightarrow r_k$ to transform $\A$\;
    \KwRet{$\A|_{{T}}==0$};

\end{algorithm}


\begin{proposition}\label{prop:alg_block_reduce_correctness}
The target block on ${T}$ can be reduced to zero in $\A$ while preserving the prior if and only if {\sc{BlockReduce}($T$)} returns true.
\end{proposition}

\vspace{0.1in}
\noindent
\textbf{Time complexity}.
First we analyze the
time complexity of {\sc TotDiagonalize} assuming that the input matrix has size $\ell\times m$.
Clearly, $\max\{\ell,m\}=O(N)$ where $N$ is the total number of generators and relations.
For each of $O(N)$ columns, we attempt to zero out every sub-column with row indices coinciding with each block $B$ of the previously determined $O(N)$ blocks. Let $B$ has $N_B$ rows. Then, the block $T$ in step 5 has $N_B$ rows and $O(N)$ columns.

To zero-out a sub-column, we create a source matrix out of $T$ which has size $O(NN_B)\times O(N^2)$ because each of
$O({N\choose 2})$ possible operations is converted to a column of size
$O(NN_B)$ in the source matrix. The source matrix $\BS$ with the target vector $c$
can be reduced with an efficient algorithm~\cite{BH74,IMH82} in $O(a+N^2(NN_B)^{\omega-1})$ time where $a$ is the total number of nonzero elements
in $[\BS|c]$ and $\omega\in[2, 2.373)$ is the exponent for matrix multiplication. We have $a=O(NN_B\cdot N^2)=O(N^3N_B)$. Therefore, for each block $B$ we spend $O(N^3N_B + N^2(NN_B)^{\omega -1})$ time in step 6. Then, observing $\sum_{B\in\B}N_B=N$, for each column we spend a total time of
\begin{equation}\label{eq:time_complexity_matrix_reduction}    
    \sum_{B\in\B} O(N^3N_B +N^2(NN_B)^{\omega-1})=O(N^4+N^{\omega+1}\sum_{B\in \B} N_B^{\omega-1})=O(N^4+N^{2\omega})=O(N^{2\omega})
\end{equation}
    Therefore, counting for all of the $O(N)$ columns, the total time for decomposition takes $O(N^{2\omega+1})$ time.

\subsection{Running {\sc TotDiagonalize} on the working example ~\ref{ex:working_example}}\label{sec:algorithm_run}


\begin{example}\label{last-example}
    Consider the binary matrix after simplication as illustrated in Example~\ref{eg:eg0_matrix}.
    $$
    \bordermatrix{\mathbf{A}  &c_1^{(1,1)}            &   c_2^{(1,2)}         &  c_3^{(2 ,1)}          \cr
                    r_1^{(0,1)} & 1    &   1 & 0                    \cr
                    r_2^{(1,0)} & 1    &   0                   & 1    \cr
                    r_3^{(1,1)} & 0                     &   1   & 1   }
    $$
    It has 4 admissible operations:
    $r_3\rightarrow r_1, r_3\rightarrow r_2, c_1\rightarrow c_2, c_1\rightarrow c_3$.
    
\begin{figure}[ht!]
\begin{center}
    \includegraphics[width=0.8\textwidth]{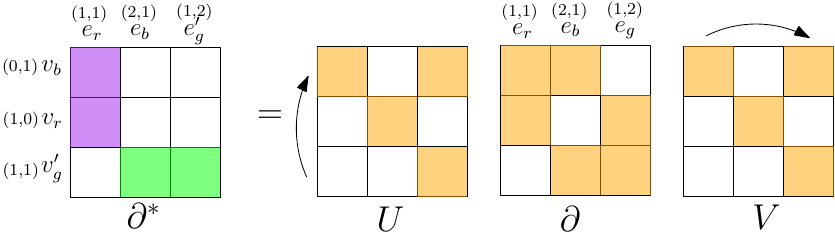}
\end{center}
\caption{
Diagonalizing the binary matrix given in Example~~\ref{eg:eg0_matrix}:
It is equivalent to multiplying the original
matrix $\partial$ with a left matrix $U$ that represents
the row operation and a right matrix $V$ that represents
the column operations.
}
\end{figure}

Before the first iteration, $\B$ is initialized to be $\B=\{B_1=(\{1\},\emptyset), B_2=(\{2\}, \emptyset), B_3=(\{3\}, \emptyset)\}$.
In the first iteration when $t=1$,
we have block $B_0=(\emptyset, \{1\})$ for column $c_1$.
For $B_1=(\{1\},\emptyset)$, the target block we hope to zero out is $T=(\{1\},\{1\})$. So we call {\sc{BlockReduce($T$)}} to check if $\A|_{T}$ can be zeroed out and update the entries on $T$ according to the results of {\sc{BlockReduce($T$)}}. 
There is only one admissible operation from outside of $T$ into it, namely, $r_3\rightarrow r_1$.  The target vector $c=\lin(\A|_{T})$ and the source matrix $\mathbf{S}=\{\lin(([\delta_{1,3}]\A)|_T)\}$ are:

\[
\kbordermatrix{ \mathbf{S}& \lin(([\delta_{1,3}]\A)|_T) & \vrule &  c=\lin(\A|_{T}) \\
 &    0 &  \vrule &  1
}
\]
The result of {\sc{ColReduce($\mathbf{S}, c$)}} stays the same as its input. That means we cannot reduce $c$ at all. Therefore, {\sc{BlockReduce($T,t$)}} returns {\sc{false}} and nothing is updated in the original matrix.

It is not surprising that the matrix remains the same because the only admissible operation that can affect $T$ does not change any entries in $T$ at all. So there is nothing one can do to reduce it, which results in merging $B_1\oplus B_0=(\{1\}, \{1\})$.
Similarly, for $B_2$ with $T=(\{2\},\{1\})$, the only admissible operation $r_3\rightarrow r_2$ does not change anything in $T$. Therefore, the matrix does not change and $B_2$ is merged with $B_1\oplus B_0$, which results in the block $(\{1,2\}, \{1\})$.
For $B_3$ with $T=(\{3\},\{1\})$, there is no admissible operation. So the matrix does not change. But $\A|_{T}=\A|_{(\{3\},\{1\})}=0$. That means {\sc{BlockReduce}} returns {\sc{true}}. Therefore, we do not merge $B_3$.
In summary, $B_0, B_1, B_2$ are merged to be one block $(\{1,2\}, \{1\})$ in the first iteration. So after the first iteration, there are two index blocks in $\B^{(1)}$:  $(\{1,2\}, \{1\})$ and $(\{3\}, \emptyset)$.

In the second iteration $t=2$, we process the second column $c_2$. Now $B_1=(\{1,2\}, \{1\}), B_2=(\{3\}, \emptyset)$ and $B_0=(\emptyset, \{2\})$.
For the block $B_1=(\{1,2\}, \{1\})$, the target block we hope to zero out is $T=(\{1,2\}, \{2\})$. There are three admissible operations from outside of $T$ into $T$, $r_3\rightarrow r_1, r_3\rightarrow r_2, c_1\rightarrow c_2$. {\sc BlockReduce}$(T)$ constructs the target vector $c=\lin(\A|_{T})$ and the source matrix $\mathbf{S}=\{\lin(([\delta_{1,3}]\A)|_T), \lin(([\delta_{2,3}]\A)|_T), \lin((\A[\delta_{1,2}])|_T)\}$ illustrated as follows:

\[
\kbordermatrix{ \mathbf{S}& \lin(([\delta_{1,3}]\A)|_T) & \lin([(\delta_{2,3}]\A)|_T) & \lin((\A[\delta_{1,2}])|_T) & \vrule &  c=\lin(\A|_{T}) \\
 &    1 & 0 & 1 & \vrule &  1 \\
 &    0 & 1 & 1 & \vrule &  0
}
\]

The result of {\sc ColReduce}$(\mathbf{S}, c)$ is
\[
\kbordermatrix{ &  &  \mathbf{S} &   & \vrule &  c \\
 &    1 & 0 & 0 & \vrule &  0 \\
 &    0 & 1 & 0 & \vrule &  0
}
\]

So the {\sc{BlockReduce}} updates $\A|_{T}$ to get the following updated matrix:
$$
\bordermatrix{\mathbf{A}'  &c_1^{(1,1)}            &   c_2^{(1,2)}        &  c_3^{(2 ,1)}   \cr
r_1^{(0,1)} +r_3^{(1,1)}    & 1    &   0    & 1  \cr
r_2^{(1,0)}                 & 1    &   0    & 1    \cr
r_3^{(1,1)}                 & 0    &   1    & 1   }
$$

and return {\sc{true}} since $\A'|_{T}==0$. Therefore, we do not merge $B_1$.
We continue to check for the block $B_2=(\{3\}, \emptyset)$ and $T=(\{3\}, \{1,2\})$, whether $\A'|_{T}$ can be reduced to zero. There is no admissible operation for this block at all. Therefore, the matrix stays the same and {\sc{BlockReduce}} returns {\sc{false}}. We merge $B_2\oplus B_0=(\{3\}, \{2\})$.

Continuing the process for the last column $c_3$ in the third iteration $t=3$, we see that $B_1=(\{1,2\}, \{1\}), B_2=(\{3\}, \{2\})$ and $B_0=(\emptyset, \{3\})$.
For the block $B_1=(\{1,2\}, \{1\})$, the target block we hope to zero out is $T=(\{1,2\}, \{2,3\})$. There are four admissible operations from outside of $T$ into $T$, $r_3\rightarrow r_1, r_3\rightarrow r_2, c_1\rightarrow c_2,  c_1\rightarrow c_3$. {\sc BlockReduce}$(T)$ constructs the target vector $c=\lin(\A|_{T})$ and the source matrix $\mathbf{S}=\{\lin(([\delta_{1,3}]\A)|_T), \lin(([\delta_{2,3}]\A)|_T), \lin((\A[\delta_{1,2}])|_T)\}, \lin((\A[\delta_{1,3}])|_T)\}$ illustrated as follows:
\[
\kbordermatrix{ \mathbf{S}& \lin(([\delta_{1,3}]\A)|_T) & \lin([(\delta_{2,3}]\A)|_T) & \lin((\A[\delta_{1,2}])|_T) & \lin((\A[\delta_{1,3}])|_T) & \vrule &  c=\lin(\A|_{T}) \\
 &    1 & 0 & 0 & 1 & \vrule &  1 \\
 &    0 & 1 & 0 & 1 & \vrule &  1 \\ 
 &    1 & 0 & 1 & 0 & \vrule &  0 \\
 &    0 & 1 & 1 & 0 & \vrule &  0
}
\]

The result of {\sc ColReduce}$(\mathbf{S}, c)$ is
\[
\kbordermatrix{ &  &  \mathbf{S} &  & & \vrule &  c \\
 &    1 & 0 & 1 & 0 & \vrule &  0 \\
 &    0 & 1 & 1 & 0 & \vrule &  0 \\
 &    1 & 0 & 0 & 0 & \vrule &  0 \\
 &    0 & 1 & 0 & 0 & \vrule &  0
}
\]

So the {\sc{BlockReduce}} updates $\A|_{T}$ to get the following updated matrix:
$$
\bordermatrix{\mathbf{A}'  &c_1^{(1,1)}            &   c_2^{(1,2)} + c_1^{(1,1)}        &  c_3^{(2 ,1)}   \cr
r_1^{(0,1)}                 & 1    &   0    & 0   \cr
r_2^{(1,0)} +r_3^{(1,1)}    & 1    &   0    & 0   \cr
r_3^{(1,1)}                 & 0    &   1    & 1   }
$$
and returns {\sc{true}} since $\A'|_{T}==0$. 
Therefore, we do not merge $B_1$ with any other block.
We continue to check for the block $B_2=(\{3\}, \{2\})$ and $T=(\{3\}, \{1,3\})$, whether $\A'|_{T}$ can be reduced to zero. There is no admissible operation for this block at all. Therefore, the matrix stays the same and {\sc{BlockReduce}} returns {\sc{false}}. We merge $B_2\oplus B_0=(\{3\}, \{2, 3\})$.

Finally the algorithm returns the matrix $\A'$ shown above as the final result. It is the correct total diagonalization with two index blocks in $\B^{\A^*}$: $B_1=(\{1,2\}, \{1\})$ and $B_2=(\{3\}, \{2,3\})$.
    An examination of {\sc ColReduce}$(\mathbf{S}, c)$ in all three iterations over columns reveals that the entire matrix $\A$ is updated by operations $r_3\rightarrow r_2$ and $c_1\rightarrow c_2$.


We can further transform it back to the original form of the presentation matrix $[\partial_1]$. Observe that a row addition $r_i\leftarrow r_i+r_j$ reverts to a basis change in the opposite direction.
$$
\bordermatrix{[\partial_1]  &e_r^{(1,1)}            &   e_b^{(1,2)}         & e_g^{(2,1)}          \cr
                v_b^{(0,1)} & \mathbf{t}^{(1,0)}    &   \mathbf{t}^{(1,1)}   & 0                    \cr
                v_r^{(1,0)} & \mathbf{t}^{(0,1)}    &   0                   & \mathbf{t}^{(1,1)}    \cr
                v_g^{(1,1)} & 0                     &   \mathbf{t}^{(0,1)}   & \mathbf{t}^{(1,0)}   }
$$
$\Longrightarrow$
$$
\bordermatrix{[\partial_1]^*  &e_r^{(1,1)}            &  e_b^{(1,2)}+\mathbf{t}^{(0,1)} e_r^{(1,1)}         &  e_g^{(2,1)}       \cr
v_b^{(0,1)}     & \mathbf{t}^{(1,0)}     &   0    & 0   \cr
v_r^{(1,0)}    & \mathbf{t}^{(0,1)}    &   0    & 0   \cr
v_g^{(1,1)}+\mathbf{t}^{(0,1)}v_r^{(1,0)}    & 0 &  \mathbf{t}^{(0,1)}    & \mathbf{t}^{(1,0)}   }
$$
\end{example}

\section{Computing presentations}\label{sec:compte_presentation}
Now that we know how to decompose a presentation by diagonalizing its
matrix form, we describe how to construct and compute these matrices in this section.
In practice, as described in Example~\ref{ex:working_example}, a persistence module is given implicitly with a simplicial filtration from which a graded module of simplicial chain complex can be inferred as we discussed before.
We always assume that the simplicial filtration is $1$-critical, which means that each simplex has a unique earliest birth time. 
For the case which is not $1$-critical, called multi-critical, one may utilize the \emph{mapping telescope},
a standard algebraic construction~\cite{AT_algtop}, which transforms a multi-critical filtration to a $1$-critical one. However, notice that this transformation  increases the input size depending on the multiplicity of the incomparable birth times of the simplices. 
For $1$-critical filtrations, each module $C_p$ is free. With a fixed basis for each free module $C_p$, a concrete matrix $[\partial_p]$ for each boundary morphism $\partial_p$ based on the chosen bases can be constructed.



With this input, we discuss our strategies for different cases that depend on two parameters, $d$, the number of parameters of filtration function, and $p$, the dimension of the homology groups in the persistence modules.

We already stated that a simplicial filtration induces a persistence module. Here,
we first give the details of this construction.
Recall that
    a ($d$-parameter) \emph{simplicial filtration} is a family of simplicial complexes $\{X_\mathbf{u}\}_{\mathbf{u}\in \Int^d}$ such that for each grade $\mathbf{u}\in \Int^d$ and each $i=1,\cdots, d$, $X_\mathbf{u}\subseteq X_{\mathbf{u}+e_i}$.
    A \emph{$d$-parameter persistence module} is a graded $R$-module where the vector spaces $M_\mathbf{u}$ are homology groups and linear maps among them are
induced by a $d$-parameter simplical filtration.

We obtain a
simplicial chain complex $(C_{\dotr}(X_\mathbf{u}), \partial_{\dotr})$ 
for each $X_\mathbf{u}$ 
in this simplicial filtration. For each comparable pairs in the grading $\mathbf{u}\leq \mathbf{v}\in\Int^d$, 
a family of inclusion maps $C_{\dotr}(X_\mathbf{u})\hookrightarrow C_{\dotr}(X_\mathbf{v})$ 
is induced by the canonical inclusion $X_\mathbf{u}\hookrightarrow X_\mathbf{v}$ giving rise to the following diagram:
\begin{center}
\begin{tikzcd}
C_{\dotr}(X_\mathbf{u}):\; \cdots \arrow[r, "\partial_{p+2}"] \arrow[d, hook] & C_{p+1}(X_\mathbf{u}) \arrow[r, "\partial_{p+1}"] \arrow[d, hook] & C_{p}(X_\mathbf{u}) \arrow[r, "\partial_{p}"] \arrow[d, hook] & C_{p-1}(X_\mathbf{u}) \arrow[r, "\partial_{p-1}"] \arrow[d, hook] & \cdots \\
C_{\dotr}(X_\mathbf{v}):\;\cdots \arrow[r, "\partial_{p+2}"]                  & C_{p+1}(X_\mathbf{v}) \arrow[r, "\partial_{p+1}"]                 & C_{p}(X_\mathbf{v}) \arrow[r, "\partial_{p}"]                 & C_{p-1}(X_\mathbf{v}) \arrow[r, "\partial_{p-1}"]                 & \cdots
\end{tikzcd}
\end{center}


For each chain complex $C_{\dotr}(X_\mathbf{u})$, we have the cycle spaces $Z_p(X_{\mathbf{u}})$'s and boundary spaces $B_p(X_{\mathbf{u}})$'s as kernels and images of boundary maps $\partial_p$'s respectively, and the homology group $H_p(X_{\mathbf{u}})=Z_p(X_{\mathbf{u}})/B_p(X_{\mathbf{u}})$ as the cokernel of the inclusion maps $B_p(X_{\mathbf{u}})\hookrightarrow Z_p(X_{\mathbf{u}})$.
In line with category theory we use the notations $\Img$, $\ker$, $\cok$ for indicating both the modules of kernel, image, cokernel and the corresponding morphisms uniquely determined by their constructions\footnote{e.g. $\ker\partial_p$ denotes the inclusion of $Z_p$ into $C_p$}. We obtain the following commutative diagram:
\begin{center}
\begin{tikzcd}
                                                                                                                  & B_p(X_{\mathbf{u}}) \arrow[r, hook] & Z_p(X_{\mathbf{u}}) \arrow[rd, "\ker \partial_p" description, hook] \arrow[r, "\cok" description, two heads] & H_p(X_{\mathbf{u}})         \\
\cdots C_{p+1}(X_{\mathbf{u}}) \arrow[rrr, "\partial_{p+1}"] \arrow[ru, "\Img{\partial_{p+1}}" description, two heads] &                                &                                                                                                          & C_p(X_{\mathbf{u}})\;\cdots
\end{tikzcd}
\end{center}

In the language of graded modules, for each $p$, the family of vector spaces and linear maps (inclusions) $(\{C_p(X_\mathbf{u})\}_{\mathbf{u}\in \Int^d}, \{C_p(X_\mathbf{u})\hookrightarrow C_p{(X_\mathbf{v})}\}_{\mathbf{u}\leq\mathbf{v}})$ can be summarized
as a $\Int^d$-graded $R$-module:
\[
C_p(X):=\bigoplus_{\mathbf{u}\in \Int^d} C_p(X_\mathbf{u}),
\mbox{ with the ring action } t_i\cdot C_p(X_{\mathbf{u}}): C_p(X_{\mathbf{u}}) \hookrightarrow C_p(X_{\mathbf{u}+e_i})\,\,\forall i, \,\forall \mathbf{u}.
\]
That is, the ring $R$ acts as the linear maps (inclusions) between pairs of vector spaces in $C_p(X_{\dotr})$ with comparable grades. It is not too hard to check that this $C_p(X_{\dotr})$ is indeed a graded module. Each $p$-chain in a chain space $C_p(X_\mathbf{u})$ is a homogeneous element with grade $\mathbf{u}$. 
Then we have a chain complex of graded modules $(C_*(X), \partial_*)$  where $\partial_*:C_*(X)\rightarrow C_{*-1}(X)$ is the boundary morphism given by $\partial_*\triangleq\bigoplus_{\mathbf{u}\in \Int^d} \partial_{*,\uu}$ with 
 $\partial_{*,\uu}:C_*(X_\mathbf{u})\rightarrow C_{*-1}(X_\mathbf{u})$ being the boundary map on $C_*(X_\mathbf{u})$.



The kernel and image of a graded module morphism are also graded modules as submodules of domain and codomain respectively whereas
the cokernel is a quotient module of the codomain. They can also be defined grade-wise in the expected way: 
\[
\mbox{For } f:M\rightarrow N, (\ker f)_{\mathbf{u}}=\ker f_\mathbf{u}, (\Img f)_\mathbf{u} = \Img f_\mathbf{u}, (\cok f)_{\mathbf{u}}=\cok f_\mathbf{u}. 
\]
All the linear maps are naturally induced from the original linear maps in $M$ and $N$.
 In our chain complex cases, the kernel and image of the boundary morphism $\partial_p: C_{p}(X)\rightarrow C_{p-1}(X)$ is the family of cycle spaces $Z_p(X)$ and family of boundary spaces $B_{p-1}(X)$ respectively with linear maps induced by inclusions. Also, from the inclusion induced morphism $B_p(X) \hookrightarrow Z_{p}(X)$, we have the cokernel module $H_p(X)$, consisting of homology groups $\bigoplus_{\uu\in\Int^d} H_p(X_{\mathbf{u}})$ and linear maps induced from inclusion maps  $X_\uu\hookrightarrow X_\vv$ for each comparable pairs $\uu\leq \vv$. This $H_p(X)$ is an example of \emph{persistence module} $M$ we mentioned in the beginning of this section, which we will study.
 It is called a persistence module $M$ because not only does it encode the information of homology groups by each graded component $M_{\mathbf{u}}$, but, roughly speaking, also tracks birth, death, merging and persistence of the homological cycles through all admissible linear maps $M_{\mathbf{u}}\rightarrow M_{\mathbf{v}}, \forall \mathbf{u}\leq \mathbf{v}$.
Classical persistence modules arising from a filtration of a simplicial complex over $\Int$ is an example of a $1$-parameter persistence module where the action
$t_1\cdot  M_\mathbf{u}\subseteq M_{\mathbf{u}+e_1}$ signifies the linear map $M_\mathbf{u}\rightarrow M_{\mathbf{v}}$ between homology groups induced by the inclusion of the complex at $\mathbf{u}$ into the complex at $\mathbf{v}=\mathbf{u}+e_1$.


In our case, we have chain complex of graded modules and induced homology groups which can be succinctly described by the following diagram:\\
\begin{center}
\begin{tikzpicture}[baseline= (a).base]
\node[scale=.9] (a) at (0,0){
\begin{tikzcd}
                                                                                                                  & B_p(X) \arrow[r, hook] & Z_p(X) \arrow[rd, "\ker(\partial_p)" description, hook] \arrow[r, two heads] & H_p(X)                                                                                      & B_{p-1}(X) \arrow[r, hook] & Z_{p-1}(X) \arrow[r, two heads] \arrow[rd, "\ker \partial_{p-1}" description, hook] & H_{p-1}(X)       \\
\cdots C_{p+1}(X) \arrow[rrr, "\partial_{p+1}"] \arrow[ru, "\Img{\partial_{p+1}}" description, two heads] &                                &                                                                                      & C_p(X) \arrow[rrr, "\partial_{p}"] \arrow[ru, "\Img \partial_{p}" description, two heads] &                                    &                                                                                             & C_{p-1}(X)\cdots
\end{tikzcd}
};
\end{tikzpicture}
\end{center}

Now we show how to compute presentations of persistence modules.

Note that a presentation gives an exact sequence $F^1\rightarrow F^0\twoheadrightarrow H\rightarrow0$.
To reveal further details of a presentation of $H$, we recognize that it respects the following commutative diagram,

\begin{center}
\begin{tikzcd}
                                                                    & Y^1 \arrow[rd, "\ker f^0" description, hook] &                                                       &  &   \\
F^1 \arrow[rr, "f^1"] \arrow[ru, "\Img f^1" description, two heads] &                                              & F^0 \arrow[rr, "f^0=\cok f^1", two heads] &  & H
\end{tikzcd}
\end{center}

where $Y^1\hookrightarrow F^0$ is the kernel of $f^0$.
With this diagram being commutative, all maps in this diagram are essentially determined by the presentation map $f^1$. We call the surjective map $f^0: F^0\rightarrow H$ \emph{generating map}, and $Y^1=\ker{f^0}$ the $1^{st}$ \emph{syzygy module} of $H$.

We introduce the following useful properties of graded modules which are used in the justifications later. They are similar to Proposition $(1.3)$ in Chapter 6 of~\cite{cox2006usingAG}.

\begin{fact}\label{fact:cox_equivalence_morphism_generators}
Let $M$ be a persistence module.
\begin{enumerate}
    \item Choosing a homogeneous element in $M$ with grade $\mathbf{u}$ is equivalent to choosing a morphism $R_{\rightarrow \mathbf{u}}\rightarrow M$.
    \item Choosing a set of homogeneous elements in $M$ with grades $\mathbf{u}_1, \cdots, \mathbf{u}_n$ is equivalent to choosing a morphism $\bigoplus_{i=1}^{n} R_{\rightarrow \mathbf{u}_i}\rightarrow M$.
    \item  Choosing a generating set of $M$ consisting of $n$ homogeneous elements with grades $\mathbf{u}_1, \cdots, \mathbf{u}_n$ is equivalent to choosing a surjective morphism
    $\bigoplus_{i=1}^{n} R_{\rightarrow  \mathbf{u}_i}\twoheadrightarrow M$.
    \item If $M\simeq \bigoplus_i R_{\rightarrow \mathbf{u}_i}$ is a free module, choosing a basis of $M$ is equivalent to choosing an isomorphism $\bigoplus R_{\rightarrow \mathbf{u}_i}\rightarrow M$.
\end{enumerate}

\end{fact}

\subsection{Multiparameter filtration, zero-dimensional homology}
In this case $p=0$ and $d>0$. This special case corresponds to determining clusters in the multiparameter setting. Importance of clusters obtained by classical one-parameter persistence has already been recognized in the literature~\cite{persis_clustering_carlsson2008persistent,liu2017visualizing}. Our algorithm computes such clusters in a multiparameter setting. In this case, we obtain a presentation matrix straightforwardly with the observation that the module $Z_0$ of cycle spaces coincides with the module $C_0$ of chain spaces.
\begin{itemize}
  \item  Presentation:
    \begin{tikzcd}
    C_1 \arrow[r, "\partial_1"] & C_0 \arrow[r, "\cok \partial_1", two heads] & H_0
    \end{tikzcd}

    \item Presentation matrix = $[\partial_1]$ is given as part of the input.
 \end{itemize}
 \noindent
 \textbf{Justification}.
 For $p=0$, the cycle module $Z_0=C_0$ is a free module. So we have the presentation of $H$ as follows:
\begin{center}
\begin{tikzcd}
                                                                            & B_0 \arrow[rd, hook] &                                             &     \\
C_1 \arrow[rr, "\partial_1" description] \arrow[ru, "\Img \partial_1" description, two heads] &                      & C_0 \arrow[r, "\cok \partial_1", two heads] & H_0
\end{tikzcd}
\end{center}
It is easy to check that $\partial_1: C_1\rightarrow C_0$ is a presentation of $H_0$ since both $C_1$ and $C_0$ are free modules. With standard basis of chain modules $C_p$'s, we have a presentation matrix $[\partial_1]$ as the valid input to our decomposition algorithm.

The $0^{th}$ homology in our working example~\ref{ex:working_example} corresponds to this case. The presentation matrix is the same as the matrix of boundary morphism $\partial_1$.



For convenience, we introduce a compact description of a presentation $f^1:F^1\rightarrow F^0$ of a module $H$. 
We write $H=<g_1,\cdots, g_n: s_1, \cdots, s_m>$ where $\{g_i\}$ is a chosen basis of $F^0$ and $\{s_j\}$ is a chosen generating set of $\Img f^1\subseteq F^0$ of $F^0$. In the working example~\ref{ex:working_example}, we can write $H_0=<v_b^{(0,1)}, v_r^{(1,0)}, v_g^{(1,1)}: \partial_1(e_r^{(1,1)}), \partial_1 (e_b^{(1,2)}), \partial_1 (e_g^{(2,1)})>$.

\subsection{2-parameter filtration, multi-dimensional homology}
\label{sec:2param}
 In this case, $d=2$ and $p\geq 0$.
 Lesnick and Wright~\cite{Lesnick_compute_minimal_present_19} have presented an algorithm to compute a presentation, in fact a minimal presentation, for this case. We restate some of their observations for completeness here.
 When $d=2$, by Hilbert Syzygy Theorem~\cite{hilbert1890ueber}, the kernel of a morphism between two free graded modules is always free.
 This implies that the canonical surjective map $Z_p\twoheadrightarrow H_p$ from free module $Z_p$ can be naturally chosen as a generating map in the presentation of $H_p$.
 In this case we have:
 \begin{itemize}
   \item Presentation:
    \begin{tikzcd}
    C_{p+1} \arrow[r, "\bar{\partial}_{p+1}"] & Z_p \arrow[r, "\cok \bar{\partial}_{p+1}", two heads] & H_p
    \end{tikzcd}
    where $\bar{\partial}_{p+1}$ is the induced map from the diagram:
    \begin{center}
    \begin{tikzcd}
                                                                                                                                                              & B_p \arrow[r, hook] & \underline{Z_p} \arrow[rd, "\ker \partial_p" description, hook] \arrow[r, two heads] & \underline{H_p} \\
    \underline{C_{p+1}} \arrow[rrr, "\partial_{p+1}"'] \arrow[ru, "\Img{\partial_{p+1}}" description, two heads] \arrow[rru, "\bar{\partial}_{p+1}"', dashed] &                     &                                                                                      & C_p
    \end{tikzcd}
    \end{center}
    \item Presentation matrix = $[\bar{\partial}_{p+1}]$ is constructed as follows:
    \begin{enumerate}
        \item Compute a basis $G(Z_p)$ for the free module $Z_p$ where $G(Z_p)$ is presented as a set of generators in the basis of $C_p$. This can be done by an algorithm in~\cite{Lesnick_compute_minimal_present_19}.
        Take $G(Z_p)$ as the row basis of the presentation matrix $[\bar{\partial}_{p+1}]$.
        \item Present $\Img\partial_{p+1}$ in the basis of $G(Z_p)$ to get the presentation matrix $[\bar\partial_{p+1}]$ of the induced map as follows. Originally, $\Img\partial_{p+1}$ is presented in the basis of $C_p$ through the given matrix $[\partial_{p+1}]$. One needs to rewrite each column of $[\partial_{p+1}]$ in the basis $G(Z_{p})$ computed in the previous step. This can be done as follows. Let $[G(Z_p)]$ denote the matrix presenting basis elements in $G(Z_p)$ in the basis of $C_p$. Let $c$ be any column vector in $[\partial_{p+1}]$. We reduce $c$ to zero vector by the matrix
        $[G(Z_p)]$ and note the columns that are added to $c$. These columns provide the necessary presentation of $c$ in the basis $G(Z_p)$.
        This reduction can be done through the traditional persistent algorithm~\cite{edelsbrunner2010computational}.
    \end{enumerate}
 \end{itemize}
\noindent
{\bf Justification}.
Unlike $p=0$ case, for $p>0$, we just know $Z_p$ is a (proper) submodule of $C_p$, which means that $Z_p$ is not necessarily equal to the free module $C_p$. However, fortunately for $d=2$, the module $Z_p$ is free, and we have an efficient algorithm to compute a basis of $Z_p$ as the kernel of the boundary map $\partial_p: C_p\rightarrow C_{p-1}$. Then, we can construct the following presentation of $H_p$:
\begin{center}
\begin{tikzcd}
           &                                                                                    & B_p \arrow[rd, hook] &                                             &               &   \\
 \arrow[r] & C_{p+1} \arrow[rr, "\bar{\partial}_{p+1}" description] \arrow[ru, "\Img \partial_{p+1}" description, two heads] &                      & Z_p \arrow[r, "\cok \bar{\partial}_{p+1}", two heads] & H_p \arrow[r] & 0
\end{tikzcd}
\end{center}
Here the $\bar{\partial}_{p+1}$ is an induced map from $\partial_{p+1}$. With a fixed basis on $Z_p$ and standard basis of $C_{p+1}$, we rewrite the presentation matrix $[{\partial}_{p+1}]$ to get $[\bar{\partial}_{p+1}]$, which constitutes a valid input to our decomposition algorithm.

\begin{figure}
    \centering
    \includegraphics[page=2, width=0.5\textwidth]{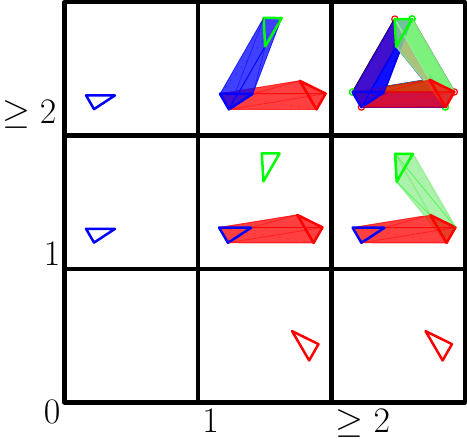}
    \caption{An example of $2$-parameter simplicial filtrations. Each square box indicates what is the current (filtered) simplical complex at the grade of the box. This example has one nontrivial cycle in $1$st homology groups at grades except $(0,0), (1,1), (2,2)$, and has two nontrivial cycles at grades $(1,1)$ and $(2,2)$. Note that all tunnels connecting triangles are hollow.}
    \label{fig:example2}
\end{figure}

\begin{example}
Consider the simplicial complex described in Figure~\ref{fig:example2}.
This is a hollow torus consisting of three empty triangles on three corners and each pair of triangles is connected by a hollow tunnel. This example is quite similar to the working example if we view the red, blue, green triangles as three generators in the $H_1$ persistence homology and three tunnels as relations connecting them. Then, we get an almost same presentation except that at grade $(2,2)$, the triangular torus introduces a new cycle which is different from any previous generators. For fixed bases of $Z_1$ and $B_1$, we can build the presentation matrix of $\bar\partial_2$. After doing some basic reduction, it can be shown that this presentation matrix is equivalent to:

$$
\bordermatrix{[\bar{\partial}_2]  &s_r^{(1,1)}            &   s_b^{(1,2)}         &  s_g^{(2,1)}          \cr
                g_b^{(0,1)} & \mathbf{t}^{(1,0)}    &   \mathbf{t}^{(1,1)}   & 0                    \cr
                g_r^{(1,0)} & \mathbf{t}^{(0,1)}    &   0                   & \mathbf{t}^{(1,1)}    \cr
                g_g^{(1,1)} & 0                     &   \mathbf{t}^{(0,1)}   & \mathbf{t}^{(1,0)}   \cr
                g_{\infty}^{(2,2)} & 0                     &   0 & 0               }
$$

where $ g_r^{(0,1)},  g_b^{(1,0)},  g_r^{(1,1)} $ represent the three triangles at the corners and $g_\infty^{(2,2)}$ represents the new cycle generated by the torus; images of $s_r^{(1,1)}         ,  s_b^{(1,2)}, s_g^{(2,1)}$ under $\bar\partial_{2}$  represent the boundaries of three tunnels.
\end{example}
 \subsection{Multiparameter filtration, multi-dimensional homology}
Now we consider the most general case where $p>0$ and $d>0$. The issue is that now $Z_p$ is not free. So, it cannot be chosen as the $0$th free module $F^0$ in the presentation of $H_p$. In what follows, we drop the index $p$ from all modules for simplicity. We propose the following procedure to construct the presentation of $H_p$. Here we use lower indices for morphisms $f_0$ and $f_1$ between free modules in presentations instead of upper indices as in $f^0$ and $f^1$ in order to write the inverse $f_i^{-1}$ of a map $f_i$ more clearly.
 \begin{itemize}
    \item

    Presentation is constructed as follows:
    \begin{enumerate}
        \item Construct a minimal presentation of $Z$ with $1^{st}$ syzygy module $Y^1$:
        \begin{center}
        \begin{tikzcd}
                                                                            & Y^1 \arrow[rd, "\ker f_0" description, hook] &                                 &   \\
        F^1 \arrow[rr, "f_1"] \arrow[ru, "\Img f_1" description, two heads] &                                              & F^0 \arrow[r, "f_0", two heads] & Z
        \end{tikzcd}
         \end{center}
        \item With the short exact sequence
        \begin{tikzcd}
        B \arrow[r, hook] & Z \arrow[r, "\pi", two heads] & H
        \end{tikzcd},
       construct the presentation of $H$:
       \begin{center}
        \begin{tikzcd}
                                                                                                                                              & f_0^{-1}(B) \arrow[rd, "\ker (\pi\circ f_0)" description, hook] &                                          &   \\
         F^1\oplus C \arrow[rr, "{\ker(\pi\circ f_0)\circ(\Img f_1+\Img\partial)}"] \arrow[ru, "{\Img f_1+\Img\partial}" description, two heads] &                                                                   & F^0 \arrow[r, "\pi\circ f_0", two heads] & H
        \end{tikzcd}
        \end{center}

        where $\pi\circ f_0$ is the composition of surjective morphisms
        \begin{tikzcd}
        F^0 \arrow[r, "f_0", two heads] & Z \arrow[r, "\pi", two heads] & H
        \end{tikzcd};
        the inclusion map $f_0^{-1}(B) \hookrightarrow F^0$
        is given by the kernel map $\ker(\pi\circ f_0)$; the surjective map $\Img f_1+\Img\partial: F^1\oplus C\twoheadrightarrow f_0^{-1}(B)$ is induced by the following diagram:
        \begin{center}
        \begin{tikzcd}
        0 \arrow[r] & F^1 \arrow[d, "\Img f_1"', two heads] \arrow[r, hook] & F^1\oplus C \arrow[r, two heads] \arrow[d, "{\exists \Img f_1+\Img\partial}" description, two heads, dashed] & C \arrow[d, "\Img\partial", two heads] \arrow[r] & 0 \\
        0 \arrow[r] & Y^1 \arrow[r, hook]                                   & f_0^{-1}(B) \arrow[r, two heads]                                                                       & B \arrow[r]                                        & 0
        \end{tikzcd}
        \end{center}
        where $\Img\partial:C\twoheadrightarrow B$ is the canonical surjective map induced from boundary map $\partial$.

        And finally, the presentation map $ F^1\oplus C\rightarrow F^0$ is just the composition $\ker(\pi\circ f_0)\circ(\Img f_1 + \Img\partial)$.
    \end{enumerate}

    Presentation matrix = $[\ker(\pi\circ f_0)\circ(\Img f_1 +\Img\partial)]$ is constructed as follows:

    \begin{enumerate}
        \item Construct a presentation matrix $[\bar{\partial}]$ the same way as in the previous case.
        \item Compute for $Y^1$ a minimal generating set $G(Y^1)$ in the basis of $G(Z)$. Let $[G(Y^1)]$ be the resulting matrix. Combine $[\bar{\partial}]$ with $[G(Y^1)]$ from right to get a larger matrix $[G(Y^1) \mid \bar{\partial}]$.
    \end{enumerate}
  \end{itemize}

\noindent
\textbf{Justification}.
First, we take a presentation of $Z$,
\begin{center}
\begin{tikzcd}
                                            & Y^1 \arrow[rd, hook] &                          &   \\
F^1 \arrow[rr, "f_1"] \arrow[ru, two heads] &                      & F^0 \arrow[r,"f_0", two heads] & Z
\end{tikzcd}
\end{center}
Here $Y^1$ is the 1st syzygy module of $Z$.
Combining it with the short exact sequence $B\hookrightarrow Z\twoheadrightarrow H$, we have,
\begin{center}
\begin{tikzcd}
f_0^{-1}(B) \arrow[r, dashed, hook] \arrow[d, two heads, dashed] & F^0 \arrow[d, "f_0", two heads] \arrow[rd, "\bar{f_0}=\pi\circ f_0", two heads, dashed] &   \\
B \arrow[r, hook]                                              & Z \arrow[r, "\pi", two heads]                                      & H
\end{tikzcd}
\end{center}
The map $\bar{f_0}=\pi\circ f_0$ is a composition of surjections and thus is a surjection from a free module $F^0$ to $H$, which is a valid candidate for the 0th free module of a presentation of $H$. Observe that the 1st syzygy module of $H$, $\ker\bar{f_0}=\ker(\pi\circ f_0)=f_0^{-1}(\ker\pi)=f_0^{-1}(B)$, and that $f_0^{-1}(B)$ can be constructed as the pullback of the maps from $B,F^0$ to $Z$. The left square commutative diagram preserves the inclusion and surjection in parallel.

Now the only thing left is to find a surjection from a free module to $f_0^{-1}(B)$. First, by the property of pullback, we know that $\ker f_0=\ker (f_0^{-1}(B)\rightarrow B)$ in a commutative way. It implies that the following diagram commutes.
\begin{center}
\begin{tikzcd}
Y^1 \arrow[d, "\ker g" description] \arrow[rd, "\ker f_0" description] &                                 \\
f_0^{-1}(B) \arrow[r, hook] \arrow[d, "g"', two heads]                 & F^0 \arrow[d, "f_0", two heads] \\
B \arrow[r, hook]                                                      & Z
\end{tikzcd}
\end{center}
Now focus on the left vertical line of the above commutative diagram. We have a short exact sequence $Y^1\hookrightarrow f_0^{-1}(B)\twoheadrightarrow B$. By the horseshoe lemma (see lemma 2.2.8 in~\cite{weibel1995introduction} for details), we can build the generating set of $f_0^{-1}(B)$ as illustrated in the following diagram:
        \begin{center}
        \begin{tikzcd}
        0 \arrow[r] & F^1 \arrow[d, "\Img f_1"', two heads] \arrow[r, hook] & F^1\oplus C \arrow[r, two heads] \arrow[d, "{\exists \Img f_1+\Img\partial}" description, two heads, dashed] & C \arrow[d, "\Img\partial", two heads] \arrow[r] & 0 \\
        0 \arrow[r] & Y^1 \arrow[r, hook]                                   & f_0^{-1}(B) \arrow[r, two heads]                                                                       & B \arrow[r]                                        & 0
        \end{tikzcd}
        \end{center}
The left projection $F^1\twoheadrightarrow Y^1$ comes from the previous presentation of $Z$. The $C\twoheadrightarrow B$ is the image map induced from boundary map $\partial:C_{p+1}\rightarrow C_p$. We take the direct sum of $F^1\oplus C$ and the horseshoe lemma indicates that there exists a projection $F^1\oplus C \twoheadrightarrow f_0^{-1}(B)$ making the whole diagram commute.
So finally, we have the valid presentation of $F^1\oplus C \rightarrow F^0\twoheadrightarrow H$.

Now we identify a generating set of $f^{-1}_0(B)$ that helps us constructing a matrix for the presentation of $H$. From the surjection $F^1\oplus C \rightarrow f^{-1}_0(B)$ in the above commutative diagram, one can see that the combination of generators from $B=\Img\partial$ and $Y^1 = \Img f_1$ forms a generating set of $ f^{-1}_0(B)$. The generators from $B=\Img\partial$ can be computed as in the previous case, which results in the matrix $[\bar\partial]$. The generators $G(Y^1)$ from $Y^1=\Img f_1$ are obtained as a result of computing the presentation of $Z$, which can be done by an algorithm presented in Skryzalin's thesis~\cite{skryzalin}. Combining these two together, we get the presentation matrix $[\bar{\partial},G(Y^1)]$ of $H$  as desired. So, now we have the solutions for all general cases. 

The above construction of presentation matrix can be understood as follows. The issue caused by non-free $Z$ is that, if we use the same presentation matrix as we did in the previous case with free $Z$, we may lose some relations coming from the inner relations of a generating set of $Z$. We fix this problem by adding these inner relations into the presentation matrix.

Figure \ref{fig:example_3} shows a simple example of a filtration of simplicial complex whose persistence module $H$ for $p=1$ is a quotient module of non-free module $Z$. The module $H$ is generated by three $1$-cycles presented as $g_1^{(0,1,1)}, g_2^{(1,0,1)}, g_3^{(1,1,0)}$. But when they appear together in $(1,1,1)$, there is a relation between these three: $\mathbf{t}^{(1,0,0)} g_1^{(0,1,1)}+ \mathbf{t}^{(0,1,0)} g_2^{(1,0,1)} + \mathbf{t}^{(0,0,1)} g_3^{(1,1,0)}=0$. Although $\Img \partial_1=0$, we still have a nontrivial relation from $Z$. So, we have
$H=<g_1^{(0,1,1)}, g_2^{(1,0,1)}, g_3^{(1,1,0)}:   s^{(1,1,1)}= \mathbf{t}^{(1,0,0)} g_1^{(0,1,1)}+ \mathbf{t}^{(0,1,0)} g_2^{(1,0,1)} + \mathbf{t}^{(0,0,1)} g_3^{(1,1,0)}>$. The presentation matrix turns out to be the following:

$$
\bordermatrix{
 & s^{(1,1,1)} \cr
g_1^{(0,1,1)} & \mathbf{t}^{(1,0,0)} \cr
g_2^{(1,0,1)} & \mathbf{t}^{(0,1,0)} \cr
g_3^{(1,1,0)} & \mathbf{t}^{(0,0,1)}
}
$$

\begin{figure}
 \centerline{\includegraphics[page=13, width=.5\linewidth]{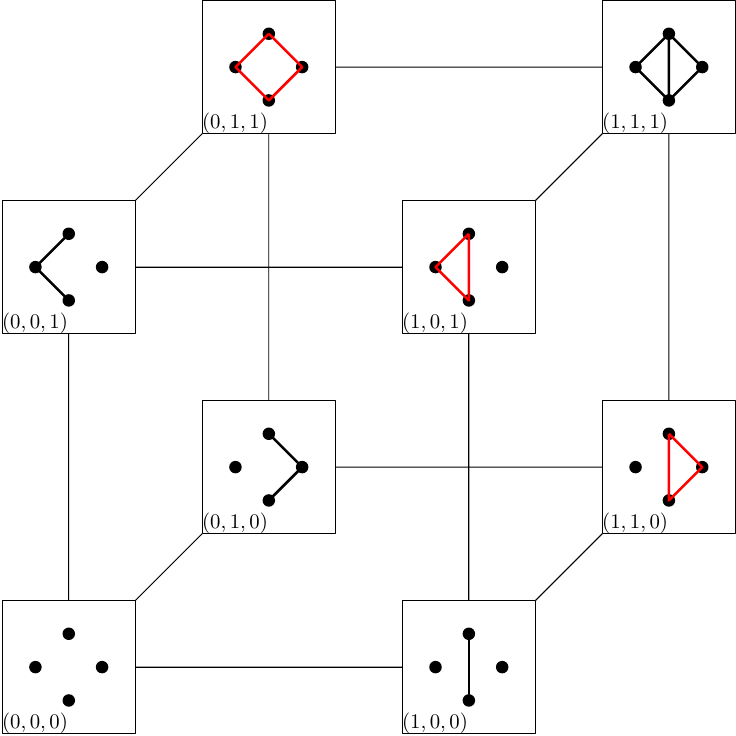}}
 \caption{An example of a filtration of simplicial complex for $d=3$ with non-free $Z$ when $p=1$. The three red circles are three generators in $Z_1$. However, at grading $(1,1,1)$, the earliest time these three red circles exist simultaneously, there is a relation among these three generators.}
 \label{fig:example_3}
\end{figure}

\subsection{Time complexity}
Now we consider the time complexity for computing presentation and decomposition together.
Let $n$ be the size of the input filtration, that is, total number of simplices obtained by counting new simplices added to the filtration (at most one new simplex at a grid point of $\mathbb{Z}^d$).
We consider three different cases as before:

\vspace{0.1in}
\noindent
\textbf{Multiparameter, $0$-dimensional homology}: In this case,
the presentation matrix $[\partial_1]$ where $\partial_1:C_1\rightarrow C_0$
has size $O(n)\times O(n)$. So, we can take $N=O(n)$ in Eqn.~\eqref{eq:time_complexity_matrix_reduction} for the time complexity analysis
of decomposition. Therefore, the total time complexity for this case is $O(n^{2\omega+1})$.

\vspace{0.1in}
\noindent
\textbf{$2$-parameter, multi-dimenisonal homology}: In this case, as described in section~\ref{sec:2param}, first we compute a basis $G(Z_p)$ that is presented in the basis of $C_p$. This is done by the algorithm of  Lesnick and Wright~\cite{Lesnick_compute_minimal_present_19} (henceforth called LW-algorithm) which runs in $O(n^3)$ time.
Using $[G(Z_p)]$, we compute the presentation matrix $[\bar\partial_{p+1}]$ as described in section~\ref{sec:2param}. This can be done in $O(n^3)$ time assuming that $G(Z_p)$ has at most $n$ elements. The presentation matrix is decomposed with {\sc TotDiagonalize} as in the previous case. However, to claim that
it runs in $O(n^{2\omega+1})$ time, one needs to ensure that the basis $G(Z_p)$ has $O(n)$ elements. This follows from the fact that $Z_p$ is a free submodule of the free module $C_p$ which has rank $n$.

\vspace{0.1in}
\noindent
\textbf{Multiparameter, multi-dimenisonal homology}:
In this case, we need to compute a generating set $G(Z)$ for $Z=Z_p$ and then a generating set $G(Y^1)$ for the 1st Syzygy module $Y^1$. Both of these generating sets can be computed by the algorithm of Skryzalin~\cite{skryzalin}(Theorem 2.6.4). The algorithm considers $O(n^{d-2})$ 
slices of $2$-parameter modules and generates the basis for each of them in
$O(n^{d+1})$ total time. This also implies that the size of the basis the algorithm computes is at most $O(n\cdot n^{d-2})=O(n^{d-1})$.

Next, we compute a generating set $G(Y^1)$ for the syzygy module $Y^1$. Recall that $Y^1\stackrel{\ker f_0}{\hookrightarrow}Z_p$ where $F_0\stackrel{f_0}{\rightarrow} Z_p$. Taking the generating set computed in the previous step as a basis for $F_0$ and observing that $\ker f_0=\ker \bar{f}_0$ where $F_0\stackrel{f_0}{\rightarrow} Z_p\stackrel{i}{\hookrightarrow} C_p$ and $\bar f_0=i\circ f_0$ we can compute a generating set $G(Y^1)$ in terms of a basis of $G(Z_p)$ using the algorithm of Skryzalin again. Again, each of
the $O(n^{d-2})$ $2$-parameter slices will generate at most $O(n)$ basis elements giving a total of $O(n^{d-1})$ basis elements. 

The matrix $[G(Y^1)]$ appended with the matrix $[\bar{\delta}_p]$, becomes the presentation matrix for $H_p$ of size $O(n^{d-1})\times O(n^{d-1})$. The decomposition algorithm on such a matrix takes at most $O(n^{(d-1)(2\omega+1)})$.

In summary, we have the following time complexity:
\begin{itemize}
    \item $d$-parameter 0-dimensional case:
        $O(n^{2\omega+1})$.
    \item $d$-parameter multi-dimensional case(general case):
       $O(n^{d+1}) + O(n^{(d-1)(2\omega+1)})=O(n^{(d-1)(2\omega+1)})$.
\end{itemize}

\section{Persistent graded Betti numbers and blockcodes} \label{sec:Persistent graded Betti numbers and blockcodes}
For $1$-parameter persistence modules, the traditional persistence algorithm computes
a complete invariant called the persistence diagram~\cite{edelsbrunner2010computational} which also has an alternative representation called barcodes~\cite{Carlsson2009}.
As a generalization of the traditional persistence algorithm, it is expected that the result of our algorithm should also lead to similar invariants. We propose two interpretations of our result as two different invariants, \emph{persistent graded Betti numbers} as a generalization of persistence diagrams and \emph{blockcodes} as a generalization of barcodes.

Both of them depend on the ideas of free resolution and graded Betti numbers which are well studied in commutative algebra and are first introduced in TDA by
Knudson~\cite{knudson2007refinement}. A brief introduction to free resolutions and their construction are given in Appendix~\ref{sec:free_resoln_graded_betti}. Here, we focus more on the two invariants mentioned above. In a nutshell, a free resolution is an extension of free presentation. Consider a free presentation of $M$ as depicted below.

\begin{center}
\begin{tikzcd}
                                                                    & Y^1 \arrow[rd, "\ker f^0" description, hook] &                                                       &  &   \\
F^1 \arrow[rr, "f^1"] \arrow[ru, "\Img f^1" description, two heads] &                                              & F^0 \arrow[rr, "f^0=\cok f^1", two heads] &  & M
\end{tikzcd}
\end{center}

If the presentation map  $f^1$ has nontrivial kernel, we can find a nontrivial map $f^2: F^2\rightarrow F^1$ with $\Img f^2=\ker f^1$, which implies $\cok f^2\cong\Img f^1=\ker f^0=Y^1$. Therefore, $f^2$ is in fact a presentation map of the first syzygy module $Y^1$ of $M$. We can keep doing this to get $f^3, f^4, \dots$ by constructing presentation maps on higher order syzygy modules $Y^2, Y^3, \dots$ of $M$, which results in a diagram depicted below, which is called a free resolution of $M$.

\begin{tikzcd}
                 &                                                                     & Y^3 \arrow[rd, "\ker f^2" description] &                                                                     & Y^2 \arrow[rd, "\ker f^1" description, hook] &                                                                     & Y^1 \arrow[rd, "\ker f^0" description, hook] &                                           &  &   \\
\cdots \arrow[r] & F^3 \arrow[rr, "f^3"] \arrow[ru, "\Img f^3" description, two heads] &                                        & F^2 \arrow[rr, "f^2"] \arrow[ru, "\Img f^2" description, two heads] &                                              & F^1 \arrow[rr, "f^1"] \arrow[ru, "\Img f^1" description, two heads] &                                              & F^0 \arrow[rr, "f^0=\cok f^1", two heads] &  & M
\end{tikzcd}

Free resolution is not unique. However, there exists an essentially unique minimal free resolution in the sense that any free resolution can be obtained by summing the minimal free resolution with a free resolution of a trivial module. 
For a graded module $M$, consider the multiset consisting of the grades of homogeneous basis elements for each $F^j$ in the minimal free resolution of $M$. We record the multiplicity of each grade $\uu\in\Int^d$ in this multiset, denoted as $\beta_{j, \uu}^M$.
Then, the mapping $\beta^M_{(-,-)}:\Int_{\geq 0}\times\Int^d \rightarrow \Int_{\geq 0}$ can be viewed as an invariant of graded module $M$, which is called the graded Betti numbers of $M$.
By applying the decomposition of module $M\simeq \bigoplus M^i$, we have for each indecomposable $M^i$, the refined graded Betti numbers $ {\beta}^{M^i}=\{\beta^{M^i}_{j,\mathbf{u}}\mid j\in \mathbb{N}, \mathbf{u}\in \Int^d\}$. We call the set $\mathcal{PB}(M):=\{\beta^{M^i}\}$ {\em persistent graded Betti numbers} of $M$. For the working example~\ref{ex:working_example}, the persistent graded Betti numbers are given in two tables listed in Table~\ref{tb:persistence_grades}.



One way to summarize the information of graded Betti numbers is to use the Hilbert function, which is also called dimension function~\cite{DX18} in TDA defined as:
$$
\dm M:\Int^d\rightarrow \Int_{\geq 0}\quad \dm M(\mathbf{u}) = \dim(M_{\mathbf{u}})
$$

\begin{fact}\label{fact:equation_dm_bettis}
There is a relation between the graded Betti numbers and dimension function of a persistence module as follows:
$$
\forall \mathbf{u}\in \Int^d, \; \dm M(\mathbf{u})=\sum_{\mathbf{v}\leq \mathbf{u}}\sum_j (-1)^j \beta_{j,\mathbf{v}}
$$
\end{fact}

Then for each indecomposable $M^i$, we have the dimension function $\dm M^i$. We call the set of dimension functions $\mathcal{B}_{\dm}(M):=\{\dm M^i\}$ the {\em blockcode} of $M$.

For our working example, the dimension functions of indecomposable summands $M^1$ and $M^2$ are:
\begin{equation}
    \dm M^1(\mathbf{u})=
    \begin{cases}
        1 & \text{if } \mathbf{u}\geq (1,0) \text{ or } \mathbf{u}\geq (0,1) \\
        0 & \text{otherwise}
    \end{cases}\qquad
    \dm M^2(\mathbf{u})=
    \begin{cases}
        1 & \text{if } \mathbf{u}=(1,1 )\\
        0 & \text{otherwise}
    \end{cases}
\end{equation}
They can be visualized
as in Figure~\ref{Fig:dm_1_2}.

The information which can be read out from graded Betti numbers and dimension functions are similar. We take the dimension functions of our working example as an example.
For $\dm M^1$, two connected components are born at the two left-bottom corners of the purple region. They are merged together immediately when they meet at grade $(1,1)$. After that, they persist forever as one connected component.
For $\dm M^2$, one connected component born at the left-bottom corner of the square green region. Later at the grades of left-top corner and right-bottom corner of the green region, it is merged with some other connected component with smaller grades of birth. Therefore, it only persists within this green region.



\begin{remark*}
    In general, both persistent graded Betti numbers and blockcodes are not sufficient to classify multiparameter persistence modules, which means they are not complete invariants. As indicated in~\cite{carlsson2009topology}, there is no complete discrete invariant for multiparameter persistence modules. However, interestingly, these two invariants are indeed complete invariants for interval decomposable modules like this example,
    which are recently studied in~\cite{botnan2016algebraic,bjerkevik2016stability,DX18}.
\end{remark*}

\begin{table}[]
    \centering
    \begin{tabular}{l||l|l|l|l|l|l|l}
               \hline\hline $\beta^{M^1}$  & (1,0) & (0,1) & (1,1) & (2,1) & (1,2) & (2,2) & $\cdots$ \\[0.5ex]  \hline\hline
    $\beta_{0}$ & 1     & 1     &      &       &   &    &    \\\hline
    $\beta_{1}$ &       &       & 1     &      &   &    & \\\hline
    $\beta_{\geq 2}$ &       &       &       &       &  & &  \\[0.5ex] \hline\hline
                $\beta^{M^2}$  & (1,0) & (0,1) & (1,1) & (2,1) & (1,2) & (2,2) & $\cdots$ \\ \hline\hline
    $\beta_{0}$ &      &      & 1     &       &   &  &  \\\hline
    $\beta_{1}$ &       &       &      &     1 & 1  & &  \\\hline
    $\beta_{2}$ &       &       &       &       &  & 1 & \\\hline
    $\beta_{\geq 3}$ &       &       &       &       &  &  &
    \end{tabular}
    \caption{Persistence grades $\mathcal{PB}(M)=\{ {\beta}^{M^1},  {\beta}^{M^2}$\}. All nonzero entries are listed in this table. Blank boxes indicate 0 entries.}
    \label{tb:persistence_grades}
\end{table}

\begin{figure}[!htb]
    \centering
     \includegraphics[width=.7\linewidth]{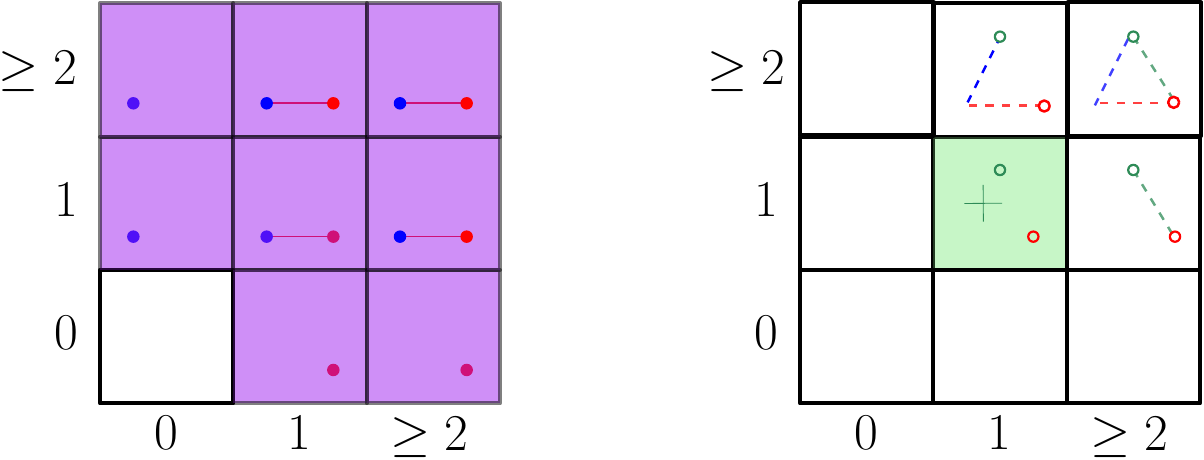}
     \caption{$\dm M^1$ and $\dm M^2$. Each colored square represents an 1-dimensional vector space $\field{k}$ and each white square represents a 0-dimensional vector space. In the left picture, $M^1$ is generated by $v_b^{0,1 }, v_r^{1,0}$ which are drawn as a blue dot and a red dot respectively. They are merged at $(1,1)$ by the red edge $e_r$. In the right picture, $M^2$ is generated by $v_g^{(1,1)}+\mathbf{t}^{(0,1)}v_r^{1,0}$ which is represented by the combination of the green circle and the red circle together at $(1,1)$. After this point $(1,1)$, the generator is mod out to be zero by relation of $e_g$ starting at $(2,1)$, represented by the green dashed line segment, and by relation of $e_b+\mathbf{t}^{(0,1)}e_r$ starting at $(1,2)$, represented by the blue dashed line segment connected with the red dashed line segment.
     }\label{Fig:dm_1_2}
\end{figure}

\subsection{Analogy with 1-parameter persistence modules}
In this section, we draw an analogy between the well known invariants, persistent diagrams and barcodes, in 1-parameter persistence modules and the invariants which we called the persistent graded Betti numbers and blockcodes respectively.

We first give an illustration of the decomposition of an $1$-parameter persistence module with a simple example.

Consider the $0^{th}$ persistence module induced by the 1-parameter simplicial filtration shown in Figure~\ref{fig:barcode}.
The $0^{th}$ homology group encodes the connected components.
From the simplicial filtration, first we can see that the number of connected components from grades 1 to 5 are $(1,2,2,1,1)$. This corresponds to the dimensions of homology vector space at each grade, which is also called the dimension function of the persistence module.
Three vertices in blue, red, and green constitute three generators denoted as $g_1,g_2,g_3$ for homology groups introduced at grades 1, 2, and 3 respectively.
In the filtration, $g_2$ is merged with $g_1$ at grade 3, and $g_3$ is merged with $g_2$ (hence also $g_1$) at grade 4.


The persistence algorithm computes the decomposition of this persistence module which results in a decomposition consisting of three indecomposable components. Each one of them corresponds to one generator. The persistence diagram summarizes the result as three pairings of grades: $ (2,3), (3,4), \mbox{ and } (1, \infty)$.
The explanations are:
\begin{enumerate}
    \item[(2,3):] The generator $g_2$ born at grade 2 is merged with a generator born earlier than it at grade 3.\\
    \item[(3,4):] The generator $g_3$ born at grade 3 is merged with a generator born earlier than it at grade 4.\\
    \item[(1, $\infty$):] The generator $g_1$ born at grade 1 is never merged with some other generator.
\end{enumerate}
The barcode represents the graph of dimension functions of each component in the decomposition. From the barcode of the example illustrated in figure~\ref{fig:barcode}, we can track directly when each generator gets born, merges (dies), and persists during its lifetime.

\begin{figure}[htb!]
    \centering
    \includegraphics[width=0.8\textwidth]{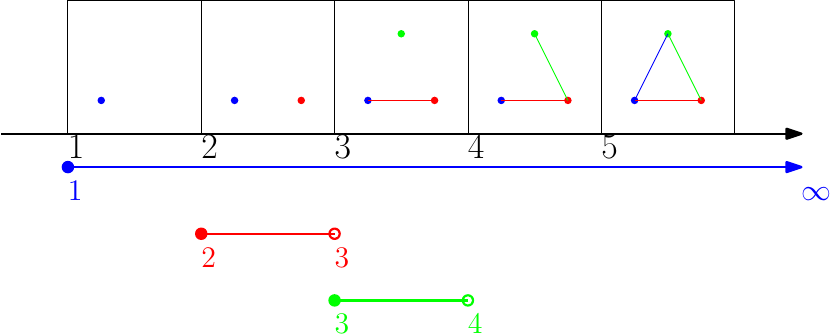}
    \caption{An example of 1-parameter simplicial filtration and its barcode.}
    \label{fig:barcode}
\end{figure}

For multiparameter persistence, we aim to compute a summary which encodes similar information as in the 1-parameter case.
Consider the simplicial bi-filtration for the working example~\ref{ex:working_example}. Similar to our example in the $1$-parameter case, we have three generators $g_{0,1}, g_{1,0}, g_{1,1}$ which are born at grades $(0,1), (1,0),(1,1)$ respectively.

As shown in~\ref{sec:algorithm_run}, the decomposition consists of two indecomposable component. One corresponds to $g_{0,1}, g_{1,0}$, and the other corresponds to $g_{1,1}$. Roughly speaking, the reason we cannot decompose $g_{0,1}$ and $g_{1,0}$ further is that their birth time are incomparable based on standard partial order of grades in $\Int^2$. When they merge together at grade $(1,1)$, neither one of them could be claimed to be merged with the other one having an earlier grade. So we have to keep them together in the same indecomposable component. However, for $g_{1,1}$, when it is merged with $g_{0,1}\mbox{ and } g_{1,0}$ at grades $(2,1)$ and $(1,2)$ respectively, both $g_{0,1}\mbox{ and } g_{1,0}$ have earlier grades than $g_{1,1}$.

Note that this explanation of decomposability based on the comparability of grades of generators does not always work as in 1-parameter case. That is why the decomposition in multiparameter case is much more complicated. But it is interesting to ask when this rule works in multiparameter case.


If we check the blockcode illustrated in Figure~\ref{Fig:dm_1_2}, we can see that for the first component, two generators are born at the two left-bottom corner of the purple region, which are grades $(0,1), (1,0)$. They are merged immediately at grade $(1,1)$. After that, none of them is merged with anything else. Therefore, the merged generator persists forever. For the second component, one generator gets born at the left-bottom corner of the green region, which is grade $(1,1)$. It persists in the green region. It is stopped by something else at grades $(1,2)$ and $(2,1)$. Therefore, it cannot persist beyond the green region.

\section{Concluding remarks}
In this paper, we propose an algorithm that generalizes the traditional persistence algorithm to the general case of multiparameter persistence. Even if its utility was clear, its design was illusive.
The results of this algorithm are interpreted as invariants we call persistent graded Betti numbers and blockcode, which can be viewed as generalizations of the persistent diagram and the barcode computed with the traditional persistence algorithm. Specifically, our algorithm can be applied to determine whether a persistence module is interval decomposable or block decomposable, which plays important roles in the computation of bottleneck distances and interleaving distances in multiparameter cases~\cite{Emerson_18_interval_decompose,botnan2016algebraic,bjerkevik2016stability,DX18}.

The assumption that no two columns nor rows have same grades is necessary for our current algorithm. 
If we consider the persistence modules induced from filtration functions in the space of all real valued functions, our assumption of distinct grading is a generic property meaning that almost all induced persistence modules satisfy the assumption. However, it is still possible that in practice the induced persistence module does not satisfy the assumption.
Without this assumption, our algorithm computes a (not necessarily total) decomposition which represents a total decomposition of some persistence module $M'$ which can be viewed as a perturbed version of the original persistence module $M$ by an arbitrarily small amount (considering $\Int^d\subseteq \Real^d$). That means, the interleaving distance between this $M'$ and $M$ is arbitrarily small. The computed decomposition $M'=\oplus M_j'$ depends on the order with which we break the ties.
How useful is this proposed strategy in practice? This question  essentially relates to the question of "stability" of decomposition structures for
which there is no satisfactory answer yet in the literature.
Currently there is no universal definition about the stability of the decomposition structure of persistence modules. A simple way to address it is to find a matching with  minimal bottleneck distance between two decomposition structures of two persistence modules with some cost function chosen for each paired indecomposable components. The most common cost function used so far is the interleaving distance. The stability-like property under this setting is an active area of recent research. 
There are some results of stability-like property on some special cases of multiparameter persistence modules, such as rectangle decomposable, triangle decomposable, block decomposable module and some other special interval decomposable modules~\cite{botnan2016algebraic}. For general multiparameter persistence modules, we know that the bottleneck distance, defined as we described above, can be much larger than interleaving distance~\cite{bjerkevik2016stability}. 
One possible solution is that, based on the stability-like property of rectangle decomposable modules and special interval decomposable modules, one can approximate the original persistence modules with rectangle or interval decomposable modules with a similar decomposition structure which may provide some stability like property~
\cite{DeyXin2021RecApprox_arXiv}.
 


We believe the two invariants that we discussed are interesting summaries containing rich information about the multiparameter persistence modules. It motivates some interesting questions for future work. What kind of new meaningful pseudo-metrics on the space of persistence modules can be constructed and computed based on these invariants, and what are the relations between the new pseudo-metrics and the existing pseudo-metrics like interleaving distance, bottleneck distance, multi-matching distance, and so on? How stable will these pseudo-metrics be?

The time complexity of our algorithm is more than $O(n^{4})$ in the $2$-parameter case. An interesting question is if one can  apply approximation techniques such as those in~\cite{dey2016simba,sheehy2013linear} to design an approximation algorithm with time complexity $o(n^4)$. We also believe that most of the techniques for speeding up computation in the traditional persistence algorithm, like those in~\cite{bauer2014distributed,bauer2017phat}, can be applied to our algorithm.

\section*{Acknowledgments.} This research is supported partially by the NSF grants CCF-1740761, CCF-2049010 and DMS-1547357. We acknowledge the influence of the
BIRS Oaxaca workshop on Multiparameter Persistence which partially seeded this work. 

\bibliographystyle{abbrv}
\bibliography{ref}




\begin{appendices}

\section{Free resolution and graded Betti numbers}\label{sec:free_resoln_graded_betti}

Here we introduce free resolutions and graded Betti numbers of graded modules. Based on these tools, we give a proof of our Theorem~\ref{thm:decomposition_correspondence}.

\begin{definition}
    For a graded module $M$, a free resolution $\mathcal{F}\rightarrow M$ is an exact sequence:

    \begin{tikzcd}
    \cdots \arrow[r] & F^2 \arrow[r, "f^2"] & F^1 \arrow[r, "f^1"] & F^0 \arrow[r, "f^0", two heads] & M \arrow[r] & 0
    \end{tikzcd}
    where each $F^i$ is a free graded $R$-module.
\end{definition}

We say two free resolutions $\mathcal{F}, \mathcal{G}$ of $M$ are isomorphic, denoted as $\mathcal{F}\simeq \mathcal{G}$, if there exists a collection of isomorphisms $\{h^i:F^i\rightarrow G^i\}_{i=0,1,\dots}$ which commutes with $f^i$'s and $g^i$'s. That is,for all $i=0,1,\dots$, $g^i \circ h^i = h^{i-1} \circ f^{i}$ where $h^{-1}$ is the identity map on $M$. See the following commutative diagram as an illustration. 

\begin{tikzcd}
    \cdots \arrow[r] & F^1 \arrow[r, "f^1"] \arrow[d, "\overset{h^1}{\simeq}"] & F^0 \arrow[r, "f^0"] \arrow[d, "\overset{h^0}{\simeq}"] & M \arrow[r] \arrow[d, "1"] & 0 \\
    \cdots \arrow[r] & G^1 \arrow[r, "g^1"]                                    & G^0 \arrow[r, "g^0"]                                    & M \arrow[r]                & 0
\end{tikzcd}

For two free resolutions $\mathcal{F}\rightarrow M$ and $\mathcal{G}\rightarrow N$, by taking direct sums of free modules $F^i\oplus G^i$ and morphisms $f^i\oplus g^i$, we get a free resolution of $M\oplus N$, denoted as $\mathcal{F}\oplus\mathcal{G}$.

Note that a presentation of $M$ can be viewed as the tail part
\begin{tikzcd}
F^1 \arrow[r, "f^1"] & F^0 \arrow[r, "f^0", two heads] & M \arrow[r] & 0
\end{tikzcd}
of a free resolution $\mathcal{F}\rightarrow M$. Free resolutions and presentations are not unique. But there exists a unique minimal free resolution in the following sense:

\begin{fact}
For a graded module $M$, there exists a unique free resolution such that $\forall i \geq 0, \Img f_{i+1}\subseteq \mathfrak{m}F_{i}$, where $\mathfrak{m}=(x_1,\cdots, x_d)$ is the unique maximal ideal of the graded ring $R=\field{k}[x_1,\cdots, x_d]$.
\end{fact}

\begin{definition}
In a minimal free resolution $\mathcal{F}\rightarrow M$,
the tail part
\begin{tikzcd}[row sep=small, column sep = small]
F^1 \arrow[r, "f^1"] & F^0 \arrow[r, "f^0", two heads] & M \arrow[r] & 0
\end{tikzcd}
is called the \emph{minimal presentation} of $M$ and $f^1$ is called the \emph{minimal presentation map} of $M$.
\end{definition}

Here we briefly state the construction of the unique free resolution without formal proof. More details can be found in~\cite{CM_rings_bruns1998,on_minimal_graded_free_resoln_romer2001dissertation}:

\begin{construction}\label{construction:appendix:free_resoln}
Choose a minimal set of homogeneous generators $g_1, \cdots, g_n$ of $M$. Let $F^0=\bigoplus_{i=1}^{n} R_{\rightarrow \gr(g_i)}$ with standard basis $e_1^{\gr(g_1)}, \cdots, e_n^{\gr(g_n)}$ of $F^0$. The homogeneous $R$-map $f^0: F^0 \rightarrow M$ is determined by $f^0(e_i)=g_i$. Now the 1st syzygy module of $M$,
\begin{tikzcd}
S_1 \arrow[r, "\ker f^0", hook] & F^0
\end{tikzcd},
is again a finitely generated graded $R$-module. We choose a minimal set of homogeneous generators $s_1, \cdots, s_m$ of $S_1$ and let $F^1=\bigoplus_{j=1}^{m} R_{\rightarrow \gr(s_j)}$ with standard basis $e_1'^{\gr(s_1)}, \cdots, e_m'^{\gr(s_m)}$ of $F^1$. The homogeneous $R$-map $f^1: F^1 \rightarrow F^0$ is determined by $f^1(e_j')=s_j$.
By repeating this procedure for $S_2=\ker f^1$ and moving backward further, one gets a graded free resolution of $M$.
\end{construction}

\begin{fact}\label{fact:minimal_free_resolution}
Any free resolution of $M$ can be obtained (up to isomorphism) from the minimal free resolution by summing it with free resolutions of trivial modules, each with the following form

\begin{tikzcd}
\cdots \arrow[r] & 0 \arrow[r] & F^{i+1} \arrow[r, "f^{i+1}=\Id"] & F^i \arrow[r] & 0 \arrow[r] & \cdots \arrow[r, two heads] & N=0 \arrow[r] & 0
\end{tikzcd}

Note that the only nontrivial morphism 
\begin{tikzcd}
F^{i+1} \arrow[r, "f^{i+1}=\Id"] & F^i
\end{tikzcd}
is the identity map $\Id$.

\end{fact}

From the above constructions, it is not hard to see that this unique free resolution is a minimal one in the sense that each free module $F^j$ has smallest possible size of basis.

For this unique free resolution, for each $j$, we can write $F^j\simeq \bigoplus_{\mathbf{u}\in \Int^d} \bigoplus^{\beta^{M}_{j,\mathbf{u}}} R_{\rightarrow \mathbf{u}}$ (the notation $\bigoplus^{\beta^{M}_{j,\mathbf{u}}} R_{\rightarrow \mathbf{u}}$ means the direct sum of ${\beta^{M}_{j,\mathbf{u}}}$ copies of $R_{\rightarrow \mathbf{u}}$).
The set $\{\beta^{M}_{j,\mathbf{u}}\mid j\in \mathbb{N}, \mathbf{u}\in \Int^d\}$
is called \emph{the graded Betti numbers} of $M$. When $M$ is clear, we might omit the upper index in Betti number.
For example, the graded Betti number of the persistence module for our working example~\ref{ex:working_example} is listed as Table~\ref{tb:graded_betti_numbers}.

\begin{table}[h!]
\centering
\begin{tabular}{l||l|l|l|l|l|l|l}
            $\beta^M$  & (1,0) & (0,1) & (1,1) & (2,1) & (1,2) & (2,2) & $\cdots$ \\ \hline\hline
$\beta_{0}$ & 1     & 1     & 1     &       &   &  &   \\\hline
$\beta_{1}$ &       &       & 1     & 1     & 1  &  & \\\hline
$\beta_{2}$ &       &       &       &       &  & 1 &  \\\hline
$\beta_{\geq 3}$ &       &       &       &       &  &  &
\end{tabular}
\caption{All the nonzero graded Betti numbers $\beta_{i,\mathbf{u}}$ are listed in the table. Empty items are all zeros.}
\label{tb:graded_betti_numbers}
\end{table}

Note that the graded Betti number of a module is uniquely determined by the unique minimal free resolution. On the other hand, if a free resolution $\mathcal{G}\rightarrow M$ with $G^j\simeq \bigoplus_{\mathbf{u}\in \Int^d} \bigoplus^{\gamma^{M}_{j,\mathbf{u}}} R_{\rightarrow \mathbf{u}}$ satisfies ${\gamma^{M}_{j,\mathbf{u}}} = {\beta^{M}_{j,\mathbf{u}}} $ everywhere, then $\mathcal{G}\simeq \mathcal{F}$
is also a minimal free resolution of $M$. 

\begin{fact}\label{fact:graded_betti_numbers_decomposition}
$\beta^{M\oplus N}_{*,*}=\beta^{M}_{*,*}+ \beta^{N}_{*,*}$
\end{fact}

\begin{proposition}\label{prop:minimal_resoln_decomposition}
Given a graded module $M$ with a decomposition $M\simeq M^1\oplus M^2$, let $\mathcal{F}\rightarrow M$ be the minimal resolution of $M$, and $\mathcal{G}\rightarrow M^1$ and $\mathcal{H}\rightarrow M^2$ be the minimal resolution of $M^1$ and $M^2$ respectively, then $\mathcal{F}\simeq \mathcal{G}\oplus \mathcal{H}$.
\end{proposition}
\begin{proof}
$\mathcal{G}\oplus\mathcal{H}\rightarrow M$ is a free resolution. We need to show it is a minimal free resolution. By previous argument, we just need to show that the graded Betti numbers of $\mathcal{G}\oplus\mathcal{H}\rightarrow M^1\oplus M^2$ coincide with graded Betti numbers of $\mathcal{F}\rightarrow M$. This is true by the fact~\ref{fact:graded_betti_numbers_decomposition}.
\end{proof}

Note that the free resolution is an extension of free presentation.
So the above proposition applies to free presentation, which immediately results in the following Corollary.

\begin{corollary}\label{cor:minimal_present_decomposition}
Given a graded module $M$ with a decomposition $M\simeq M^1\oplus M^2$, let $f$ be its minimal presentation map, and $g$, $h$ be the minimal presentation maps of $M^1, M^2$ respectively, then $f\simeq g\oplus h$.
\end{corollary}

We also have the following fact relating morphisms:
\begin{fact}
$\ker (f^1\oplus f^2)=\ker f^1\oplus \ker f^2$;
$\cok (f^1\oplus f^2)=\cok f^1\oplus \cok f^2$.
\end{fact}

Based on the above statements, now we can prove Theorem~\ref{thm:decomposition_correspondence}


\begin{proof}[proof of Theorem~\ref{thm:decomposition_correspondence}]

With the obvious correspondence $[f_i]\leftrightarrow [f]_i$, 
($2\leftrightarrow3$) easily follows from our arguments about matrix diagonalization in the main context.

($1\rightarrow 2$) Given $H\simeq \bigoplus H^i$ with the minimal presentation maps $f$ of $H$: For each $H^i$, there exists a minimal presentation map $f_i$. By Corollary~\ref{cor:minimal_present_decomposition}, we have $f\simeq \bigoplus f_i$.

($2\rightarrow 1$) Given $f\simeq \bigoplus f_i$: Since $H=\cok f= \cok(\bigoplus f_i)=\bigoplus \cok f_i$, let $H^i=\cok f_i$, we have the decomposition $H\simeq \bigoplus H^i$.

It follows that the above two constructions together give the desired 1-1 correspondence.
\end{proof}

    

\begin{proof}[proof of Proposition~\ref{prop:decomposition_correspondence}]
    We start with (2). 
    Consider the total decomposition $f\simeq \bigoplus f^i$. By Remark~\ref{rmk:presentation_and_minimal}, any presentation is isomorphic to a direct sum of the minimal presentation and some trivial presentations. Let $f\simeq g\oplus h$ with $g$ being the minimal presentation. So $\cok h=0$. Let $g\simeq \bigoplus g^j$ and $h\simeq \bigoplus h^k$ be the total decomposition of $g$ and $h$. Note that $\forall k, \cok h^k=0$.
    Now we have $\cok f\simeq \bigoplus \cok f^i$ with $\cok f^i$ being either $\cok g^j$ or $0$, by the essentially uniqueness of total decomposition. 
    With $H\simeq \bigoplus \cok g^j$ being a total decomposition of $H$ by Remark~\ref{rmk:total_decomposition_correspondence}, and $\bigoplus \cok f^i=\bigoplus \cok g^j \bigoplus 0$, we can say that $H\simeq \bigoplus \cok f^i$ is also a total decomposition. 

    
    Now for (1). For any decomposition $H\simeq \bigoplus H^i$, it is not hard to see that each $H^i$ can be written as a direct sum of a subset of $H_*^j$'s with $H\simeq \bigoplus H_*^j$ being the total decomposition of $H$. One just need to combine the  $f^i$'s correspondingly in the total decomposition of $f\simeq \bigoplus f^i$ to get the desired decomposition of $f$.
\end{proof}

\section{Missing proofs in Section~\ref{sec:computing_decomposition}}\label{sec:missing_proofs}

\begin{proposition*}[\ref{prop:block_reduction_linearization}]
    The target block $\A|_{{T}}$ can be reduced to 0 while preserving the prior if and only if $\A|_{{T}}$ can be written as a linear combination of independent operations. That is,
    \begin{equation}
        \A|_{{T}}=
    \sum_{\mathclap{\substack{l\notin\row(T)\\ k\in\row({T}) }}}\alpha_{k,l} \X^{k,l}|_{{T}}+\sum_{\mathclap{\substack{i\notin\col(T)\\ j\in \col({T}) }}} \beta_{i,j}\Y^{i,j}|_{{T}} 
    \label{eq:ops_appendix}
    \end{equation}
    where 
    $\alpha_{k,l}$'s and $\beta_{i,j}$'s are coefficient in $\mathbb{k}=\mathbb{F}_2$. 
\end{proposition*}

\begin{proof}

Everything in the statement of the proposition is restricted to ${T}$.
For simplicity of notations, we omit the lower script ${\leq t}$ by assuming $\A_{\leq t}=\A$, i.e., $t=m$ is the last column index. 
It can be verified that this omission does not affect the proof.
The simple reason is that because of the admissible rules of column operations,  entries beyond column $t$ carried by any admissible operations will never affect entires in $\A_{\leq t}$.

Recall that 
$\Y^{i,j}=\A\dotr[\delta_{i,j}]$ for some $(i,j)\in \colop$ and $\X^{k,l}=[\delta_{k,l}]\dotr\A$ for some $(l,k)\in \rowop$ where
\[
\colop=\{(i,j)\mid c_i\rightarrow c_j \mbox{ \textrm{ is an admissible column operation}} \}\subseteq \col(\A)\times \col(\A)  \mbox{ and }
\]
\[
\rowop=\{(l,k)\mid r_l\rightarrow r_k \mbox{ \textrm{ is an admissible row operation}} \}\subseteq \row(\A)\times \row(\A)
\]

Let $\I$ be the identity matrix. We say a matrix $\LL$ is an admissible left multiplication matrix if $\LL=\I+\sum_{\rowop} \alpha_{k,l}[\delta_{k,l}]$ for some $(l,k)\in\rowop, \alpha_{k,l}\in \{0,1\}$. Similarly, we say a matrix $\RR$ is an admissible right multiplication matrix if $\RR=\I+\sum_{\mathcal{\colop}} \beta_{i,j}[\delta_{i,j}]$ for some $(i,j)\in\colop, \beta_{i,j}\in \{0,1\}$.
In short, we just say $\LL$ and $\RR$ are admissible.

It is not difficult to observe the following properties of admissible matrices:

\begin{fact}
Matrix $\A'\sim \A$ is an equivalent matrix transformed from $\A$ by a sequence of admissible operations if and only if $\A'=\LL\A\RR$ for some admissible $\LL$ and $\RR$.
\end{fact}

\begin{fact} \label{fact:admissible_inverse_and_multiplication}
 Admissible matrices are closed under multiplication and taking inverse.
\end{fact}

\begin{fact}\label{fact:admissible_submatrix_invertible}
For any admissible $\LL$, let $S\subseteq \row(\LL)$ be any subset of row indices. Then $\LL|_{S\times S}$ is invertible.
\end{fact}
For the last fact, observe that the matrix $\LL|_{S\times S}$ can be embedded as a block of an admissible matrix $\LL'$ constructed by making all off-diagonal entries of $\LL$ whose indices are not in $S\times S$ to be zero.
The matrix $\LL'$ is obviously admissible. So by the second fact, it is invertible. Also, $\LL'$ can be written in block diagonal form with two blocks $\LL'|_{S\times S} \mbox{ and } \LL'|_{\bar{S}\times \bar{S}}=\I$ where $\bar{S}=\row(\LL')-S$.  Therefore, if $\LL'$ is invertible, so is $\LL|_{S\times S}= \LL'|_{S\times S}$.

\begin{figure}[!htb]
        \centering
         \includegraphics[width=\linewidth, page=2]{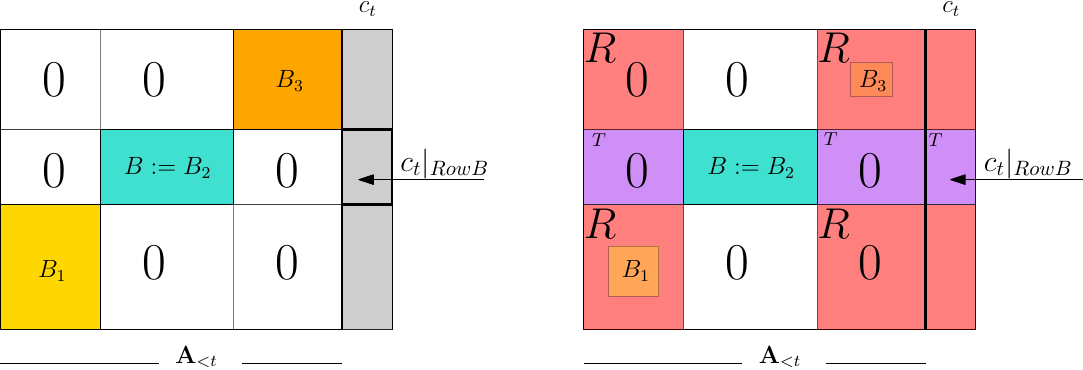}
         \caption{ (Left) $\A$ at iteration $t$ during reduction of the sub-column 
            $c_t|_{\row(B)}$ for the block $B=B_2$. (Right) Target block ${T}$ shown in magenta includes the sub-column of $c_t$. It does not include $B:=B_2$. All rows external to $T$ have zeros in the columns external to ${T}$. All columns external to ${T}$ have zeros in the rows external to ${T}$. Red regions combined form $R$.
         }\label{Fig:incremental}
\end{figure}
We write the matrix $\A$ in the following block forms with respect to $B$ and $T$ with necessary reordering of rows and columns (See Figure~\ref{Fig:incremental} for a simple illustration without reordering rows and columns):

\[
\A =
\left[
\begin{array}{c|c}
R & 0 \\
\hline
T & B
\end{array}
\right]
\]

Here we abuse the notations of block and index block to make the expression more legible. In the above block forms of $\A$, for example, $T$ represents the entries of $\A$ on the index block $T$, that is the block $\A|_{T}$, which is the target block we want to reduce. Note that
\begin{equation*}
\begin{split}
R&=\big[\row(\A)\setminus\row(B),\, \col(\A_{\leq t})\setminus\col(B))\big]\\
&=[\bigoplus_{B_i\neq B} B_i]\cup \big[\row(\A)\setminus\row(B),\, \{t\}\big]
\end{split}
\end{equation*}
which is the block obtained by merging all other previous index blocks together with the sub-column of $t$ excluding entries on $\row(B)$.
The right top block is zero since it belongs to the intersections of rows and columns from different  blocks.


Observe that, the target block $T$ can be reduced to 0 in $\A$ with prior preserved if and only if

\begin{equation}\label{eq:LAR_inverse}
\LL\A \RR:=
    \LL
    \dotr
    \left[
\begin{array}{c|c}
R & 0 \\
\hline
T & B
\end{array}
\right]
\dotr
\RR=
    \left[
\begin{array}{c|c}
R & 0 \\
\hline
0 & B
\end{array}
\right]
\end{equation}
for some admissible $\LL$ and $\RR$.

For $\Leftarrow$ direction, consider 
$\LL= \I+\sum \alpha_{k,l}[\delta_{k,l}]$ and 
$\RR= \I+\sum \beta_{i,j}[\delta_{i,j}]$ 
with binary coefficients $\alpha_{k,l}$'s and $\beta_{i,j}$'s given in Equation~\ref{eq:ops_appendix}. Then, we have
\begin{align}
    \LL\A\RR &= (\I+\sum \alpha_{k,l}[\delta_{k,l}]) \A (\I+\sum \beta_{i,j}[\delta_{i,j}])\\ 
    &= \A + \sum \alpha_{k,l}[\delta_{k,l}]\A + \sum \beta_{i,j}\A[\delta_{i,j}] + 
    \sum \sum \alpha_{k,l}\beta_{i,j}[\delta_{k,l}]\A[\delta_{i,j}] \\ 
    &= \A + \sum \alpha_{k,l}[\delta_{k,l}]\A + \sum \beta_{i,j}\A[\delta_{i,j}] \label{eq:third_eq}\\
    &=\A + \sum \alpha_{k,l}\X^{k,l} + \sum \beta_{i,j}\Y^{i,j} \label{eq:obs2}
\end{align}
The third Equation~(\ref{eq:third_eq}) follows from Observations~\ref{prop:cross_term_0}. After restriction to $T$, by the assumption that $\sum \alpha_{k,l}\X^{k,l} + \sum \beta_{i,j}\Y^{i,j}=\A|_T$, we get $\LL\A\RR|_T=0$. By the definition of independent operations and Observation~\ref{obs:prior}, one can see that our $\LL, \RR$ solves Equation~\ref{eq:LAR_inverse}. 

For $\Rightarrow$, 
we will show that if the above equation is solvable, then there always exist solutions $\LL'$ and $\RR'$ in a simpler forms as stated in the following proposition.

\begin{proposition}\label{lm:linear_comb}
Equation~(\ref{eq:LAR_inverse}) is solvable for some admissible $\LL$ and $\RR$ if and only if it is solvable for some admissible $\LL'$ and $\RR'$ in the following form:

\begin{equation}\label{eq:LLRRsimpleform}
\LL'=
    \left[
\begin{array}{c|c}
I & 0 \\
\hline
U & I
\end{array}
\right]
\mbox{ and }
\RR'=
    \left[
\begin{array}{c|c}
I & 0 \\
\hline
V & I
\end{array}
\right]
\end{equation}

\end{proposition}

Before we prove Proposition~\ref{lm:linear_comb}, we show how one can prove the $\Rightarrow$ direction in Proposition~\ref{prop:block_reduction_linearization} from it. Based on the equivalent condition Equation~\ref{eq:LAR_inverse} and Proposition~\ref{lm:linear_comb}, 
we can write $\LL'$ and $\RR'$ in formula~\ref{eq:LLRRsimpleform} as

\begin{equation*}
\LL'=\I+\sum_{\mathclap{\substack{(l,k)\in\\ \rowop_{R\rightarrow T}}}} \alpha_{k,l}[\delta_{k,l}]
\;\quad
\RR'=\I+\sum_{\mathclap{\substack{(i,j)\in \\ \colop_{B\rightarrow T}}}} \beta_{i,j}[\delta_{i,j}]
\end{equation*}
where $\rowop_{R\rightarrow T}=\{(l,k)\in \rowop\mid (l,k)\in \row(R)\times \row(T)\}$ and $\colop_{B\rightarrow T}=\{(i,j)\in \colop\mid (i,j)\in \col(B)\times \col(T)\}$, and  $\alpha_{k,l}, \beta_{i,j}\in \{0,1\}$. Then, similar to Equation~\ref{eq:obs2}, we get 
\begin{equation*}
\begin{split}
\LL'\A\RR'&=(\I+\sum_{\mathclap{\substack{(l,k)\in\\ \rowop_{R\rightarrow T}}}} \alpha_{k,l}[\delta_{k,l}])\dotr\A\dotr(\I+\sum_{\mathclap{\substack{(i,j)\in \\ \colop_{B\rightarrow T}}}} \beta_{i,j}[\delta_{i,j}])\\
 &=\A+\sum_{\mathclap{\substack{(l,k)\in\\ \rowop_{R\rightarrow T}}}}\alpha_{k,l} [\delta_{k,l}]\dotr \A+\sum_{\mathclap{\substack{(i,j)\in \\ \colop_{B\rightarrow T}}}} \beta_{i,j}\A\dotr [\delta_{i,j}] + \sum_{{\substack{(l,k)\in\\ \rowop_{R\rightarrow T}}}}\sum_{{\substack{(i,j)\in \\ \colop_{B\rightarrow T}}}} \alpha_{k,l}\beta_{i,j} [\delta_{k,l}]\dotr \A\dotr [\delta_{i,j}]\\
 &=\A+\sum_{\mathclap{\substack{(l,k)\in\\ \rowop_{R\rightarrow T}}}}\alpha_{k,l} [\delta_{k,l}]\dotr \A+\sum_{\mathclap{\substack{(i,j)\in \\ \colop_{B\rightarrow T}}}} \beta_{i,j}\A\dotr [\delta_{i,j}]\\
 &=\A+\sum_{\mathclap{\substack{(l,k)\in\\ \rowop_{R\rightarrow T}}}}\alpha_{ k,l }\X^{k,l}+\sum_{\mathclap{\substack{(i,j)\in \\ \colop_{B\rightarrow T}}}} \beta_{i,j}\Y^{i,j}
\end{split}
\end{equation*}



By restriction on $T$ we have
\begin{equation}\label{eq:after_restriction_T}
    \LL'\A\RR'|_{T}=
    =\A|_{T}+\sum_{\mathclap{\substack{(l,k)\in\\ \rowop_{R\rightarrow T}}}}\alpha_{ k,l }\X^{k,l}|_{T}+\sum_{\mathclap{\substack{(i,j)\in \\ \colop_{B\rightarrow T}}}} \beta_{i,j}\Y^{i,j}|_{T}
\end{equation}
With $\LL'\A\RR'|_{T}=0$ by our assumption, we get
\begin{equation*}
    \A|_{T}=\sum_{\mathclap{\substack{(l,k)\in\\ \rowop_{R\rightarrow T}}}}\alpha_{ k,l } \X^{ k,l }|_T+\sum_{\mathclap{\substack{(i,j)\in \\ \colop_{B\rightarrow T}}}} \beta_{i,j}\Y^{i,j}|_T
\end{equation*}
This is exactly what we want
\begin{equation}
    \A|_{{T}}=
\sum_{\mathclap{\substack{l\notin\row(T)\\ k\in\row({T}) }}}\alpha_{k,l} \X^{k,l}|_{{T}}+\sum_{\mathclap{\substack{i\notin\col(T)\\ j\in \col({T}) }}} \beta_{i,j}\Y^{i,j}|_{{T}} 
\end{equation}





\end{proof}

Now we give the proof of Proposition~\ref{lm:linear_comb}.
\begin{proof}[Proof of Proposition~\ref{lm:linear_comb}]
The $\Leftarrow$ direction is trivial.  
For the $\Rightarrow$ direction, 
we want to show that, if Equation~(\ref{eq:LAR_inverse}) is solvable for some admissible $\LL$ and $\RR$, then there exists admissible $\LL'$ and $\RR'$ so that
\[
\LL'=
\left[
\begin{array}{c|c}
I & 0 \\
\hline
U & I
\end{array}
\right],
\RR'=
    \left[
\begin{array}{c|c}
I & 0 \\
\hline
V & I
\end{array}
\right],
\mbox{ and }
\LL^{'}\dotr
    \left[
\begin{array}{c|c}
R & 0 \\
\hline
T & B
\end{array}
\right]
\dotr
\RR^{'}=
 \left[
\begin{array}{c|c}
R & 0 \\
\hline
UR+BV+T & B
\end{array}
\right]=
\left[
\begin{array}{c|c}
R & 0 \\
\hline
0 & B
\end{array}
\right]
\]

We write $\LL$ and $\RR$ in corresponding block forms as follows:
\begin{equation}
\LL=
 \left[
\begin{array}{c|c}
P_1 & P_2 \\
\hline
P_3 & P_4
\end{array}
\right]
\mbox{ and }
\RR=
 \left[
\begin{array}{c|c}
Q_1 & Q_2 \\
\hline
Q_3 & Q_4
\end{array}
\right]
\end{equation}

From Equation~(\ref{eq:LAR_inverse}) one can get a set of equations
\begin{equation}\label{eq:LAR_0}
    P_1 R Q_2 + P_2 B Q_4 = 0
\end{equation}
\begin{equation}\label{eq:LAR_A}
    P_1 R Q_1 + P_2 B Q_3 = R
\end{equation}
\begin{equation}\label{eq:LAR_C}
    P_3 R Q_2 + P_4 B Q_4 = B
\end{equation}
\begin{equation}\label{eq:LAR_B}
    P_3 R Q_1 + P_4 B Q_3 = T
\end{equation}
From Fact~\ref{fact:admissible_submatrix_invertible}, we know that $P_1, P_4, Q_1, Q_4$ are invertible.
By left multiplication with $P_1^{-1}$ and right multiplication with $Q_4^{-1}$ on both sides of Equation~(\ref{eq:LAR_0}), one can get :
\begin{equation} \label{eq:LAR_01}
    P_1^{-1}P_1 R Q_2 Q_4^{-1} + P_1^{-1} P_2 B Q_4 Q_4^{-1} = R Q_2 Q_4^{-1} + P_1^{-1} P_2 B = 0 \implies - R Q_2 Q_4^{-1} =  P_1^{-1} P_2 B
\end{equation}
Similarly, by left multiplication with $P_1^{-1}$ on both sides of Equation~(\ref{eq:LAR_A}) and by right multiplication with $Q_4^{-1}$ on both sides of Equation~(\ref{eq:LAR_C}), one can get the following equations:
\begin{equation}\label{eq:LAR_A1}
        P_1 R Q_1 + P_2 B Q_3 = R \implies   R Q_1  = P_1^{-1}R - P_1^{-1}P_2 B Q_3
\end{equation}
\begin{equation}\label{eq:LAR_C1}
        P_3 R Q_2 + P_4 B Q_4 = B \implies
         P_4 B = B Q_4^{-1} - P_3 R Q_2  Q_4^{-1}
\end{equation}
Now from Equation~\ref{eq:LAR_B}, we have:
\begin{equation*}
\begin{split}
                    &  T  =P_3 R Q_1 + P_4 B Q_3  \\
\xRightarrow[]{\mbox{Equation~\ref{eq:LAR_A1} and ~\ref{eq:LAR_C1}}} \qquad  & T  = P_3 (P_1^{-1}R - P_1^{-1}P_2 B Q_3) + (B Q_4^{-1} - P_3 R Q_2  Q_4^{-1})Q_3\\
& = P_3 P_1^{-1}R + B Q_4^{-1}Q_3 - P_3 P_1^{-1}P_2 B Q_3 - P_3 R Q_2  Q_4^{-1}Q_3 \\
\xRightarrow[]{\mbox{Equation~\ref{eq:LAR_01}}} \qquad   & T =P_3 P_1^{-1}R + B Q_4^{-1}Q_3 - P_3 P_1^{-1}P_2 B Q_3 + P_3 P_1^{-1} P_2 B Q_3 \\
& = P_3 P_1^{-1}R + B Q_4^{-1}Q_3
\end{split}
\end{equation*}
Letting $U=P_3 P_1^{-1}$ and $V=Q_4^{-1}Q_3$, we get the desired equation.
Now we just need to show that $\LL', \RR'$ are both admissible. We prove it for $\RR'$. Similar proof holds for $\LL'$.
We want to show that for any $(i,j)\in \row(V)\times\col(V)$, if $\RR'_{i,j}=1$, then $(i,j)\in \colop$.
From equality, $V=Q_4^{-1}Q_3$, which implies $\RR'_{i,j}=\sum_k (Q_4^{-1})_{i,k} \dotr (Q_3)_{k,j}=1$, we know that $(Q_4^{-1})_{i,k}= (Q_3)_{k,j}=1$ for some $k$. Since $Q_4^{-1} \mbox{ and } Q_3$ are both blocks in the admissible matrix $\RR$, by the definition of admissible left multiplication matrix, we have $(i,k), (k,j)\in \colop$. Note that $\colop$ is closed under transitive relation by Proposition~\ref{prop:transitive relation}. So we have $(i,j)\in \colop$.

\end{proof}

\end{appendices}

\end{document}